\documentclass[10pt,reqno]{amsart}
\usepackage{amssymb,mathrsfs,graphicx,extpfeil,amsmath}
\usepackage{ifthen,latexsym,float,colortbl}
\usepackage[margin=1in]{geometry}
\usepackage{caption}
\usepackage{sidecap}
\usepackage{hyperref}
\usepackage{rotating,dsfont,stackengine}
\usepackage{colortbl}
\definecolor{black}{rgb}{0.0, 0.0, 0.0}
\definecolor{red}{rgb}{1.0, 0.5, 0.5}
\provideboolean{shownotes} 
\setboolean{shownotes}{true} 
\newcommand{\margnote}[1]{
\ifthenelse{\boolean{shownotes}}%
{\marginpar{\raggedright\tiny\texttt{#1}}}%
{}%
}
\newcommand{\hole}[1]{
\ifthenelse{\boolean{shownotes}}%
{\begin{center} \fbox{ \rule {.25cm}{0cm} \rule[-.1cm]{0cm}{.4cm}
\parbox{.85\textwidth}{\begin{center} \texttt{#1}\end{center}} \rule
{.25cm}{0cm}}\end{center}} {} }

\graphicspath{{pics/}}

	\title[Global existence and asymptotic stability for the TT model]{Global existence and asymptotic stability for the Toner--Tu model of flocking}
	\author{Young-Pil Choi}
\address[Young-Pil Choi]{\newline Department of Mathematics \newline
Yonsei University, 50 Yonsei-Ro, Seodaemun-Gu, Seoul 03722, Republic of Korea}
\email{ypchoi@yonsei.ac.kr}

\author{Kyungkeun Kang}
\address[Kyungkeun Kang]{\newline Department of Mathematics \newline
Yonsei University, 50 Yonsei-Ro, Seodaemun-Gu, Seoul 03722, Republic of Korea}
\email{kkang@yonsei.ac.kr}

\author{Woojae Lee}
\address[Woojae Lee]{\newline Department of Mathematics \newline
Yonsei University, 50 Yonsei-Ro, Seodaemun-Gu, Seoul 03722, Republic of Korea}
\email{woori0108@yonsei.ac.kr}

\numberwithin{equation}{section}

\newtheorem{theorem}{Theorem}[section]
\newtheorem{lemma}{Lemma}[section]

\newtheorem{remark}{Remark}[section]

\newcommand{\R}{\mathbb R}

\newcommand{\N}{\mathbb N}

\newcommand{\ls}{\lesssim}

\newcommand{\bq}{\begin{equation}}
\newcommand{\eq}{\end{equation}}
\newcommand{\e}{\varepsilon}
\newcommand{\lt}{\left}
\newcommand{\rt}{\right}

\newcommand{\pa}{\partial}

\makeatletter
\def\moverlay{\mathpalette\mov@rlay}
\def\mov@rlay#1#2{\leavevmode\vtop{%
   \baselineskip\z@skip \lineskiplimit-\maxdimen
   \ialign{\hfil$\m@th#1##$\hfil\cr#2\crcr}}}
\newcommand{\charfusion}[3][\mathord]{
    #1{\ifx#1\mathop\vphantom{#2}\fi
        \mathpalette\mov@rlay{#2\cr#3}
      }
    \ifx#1\mathop\expandafter\displaylimits\fi}
\makeatother

\begin{document}
\allowdisplaybreaks

\date{\today}

\subjclass[]{}

\begin{abstract}
This paper deals with the Toner--Tu (TT) model, which is a hydrodynamic model describing the collective motion of numerous self-propelled agents. We analytically study the global-in-time well-posedness of the TT model near the steady-state solution in the ordered phase. We also show the large-time behavior of solutions showing that the steady-state solution is polynomially stable in a Sobolev space in the sense that solutions that are initially close to that steady state converge to that at least polynomially fast as time tends to infinity. Moreover, we investigate the variant of the TT model which describes the dynamics of the actin filament. 
\end{abstract}

\maketitle \centerline{\date}

\tableofcontents

%
%
%
%
%
%
%
	\section{Introduction}

Active matter encompasses both living and non-living systems characterized by the presence of self-propelled particles, each capable of converting energy into coordinated motion. Theories surrounding active matter have made significant strides in elucidating the dynamics of suspensions composed of active polar particles, exemplified by phenomena like fish schooling, locust swarming, and bird flocking, see for instance \cite{CDP09, CCP17, CHL17, CDMBC07, CS07, DFMT19, Deg18, DP17, DCBC06, marchetti2013hydrodynamics, Tad21} and references therein. Vicsek et al.  \cite{vicsek1995novel} pioneered a model describing the collective motion of numerous organisms, a phenomenon known as {\it flocking}, within a nonequilibrium dynamical system. Subsequently, Toner and Tu further refined this model into a continuum framework, leveraging conservation laws and symmetry considerations from the Vicsek model \cite{toner1995long,toner1998flocks}. Various polar particle suspension systems related to the Toner-Tu (TT) model have been instrumental in explaining pattern formation in systems like actin filaments and myosin \cite{le2016pattern, husain2017emergent}, as well as explaining turbulence phenomena observed in bacterial colonies   \cite{wensink2012meso, linkmann2019phase} within active fluids. The pivotal insight offered by these models lies in their ability to observe a phase transition from disorder to order within these systems.	
	
In the seminal work \cite{toner1995long} of Toner and Tu, they delved into the nonequilibrium phase transition towards collective motion by formulating mesoscopic equations using a renormalization group approach in two dimensions. They further extended this study to a generalized model in the absence of Galilean invariance, elucidated in \cite{toner1998flocks}. Subsequently, Bertin et al. \cite{bertin2006boltzmann} and Ihle \cite{ihle2011kinetic} independently derived the deterministic TT model from the Boltzmann equation. These derivations have facilitated the analysis of microscopic phenomena through the lens of hydrodynamics. Numerical investigations of the TT model have been conducted under various conditions, including incompressibility \cite{chen2016mapping} and constant density \cite{besse2022metastability}.
	   
In the current work, we analytically study the global-in-time well-posedness of the TT model near the steady-state solution in the ordered phase and its asymptotic stability. Consider the TT model proposed in \cite{toner1995long}:
\begin{align}\label{original tt model}
\begin{split}
&v_t+\lambda v\cdot\nabla v=\alpha v-\beta \vert v\vert^2v+D_1\nabla{\rm {div}}\  v+D_2\Delta v-\nabla P+D_L(v\cdot\nabla)^2v +f,\\
&\rho_t+{\rm{div}}\  (v \rho)=0,
\end{split}
\end{align}
where  $\rho=\rho(x,t)$ represents the number density and $v=v(x,t)$ denotes the velocity fields at position $x\in\mathbb{R}^d$ and time $t>0$ with $d=2,3$. The parameters $\beta, D_1, D_2, D_L$ are positive constants and the $\lambda$ is a dimensionless parameter. Additionally, $\alpha<0$ characterizes the disordered phase, whereas $\alpha>0$ signifies the ordered state, delineating the phase transition. Here, $f=f(x,t)$ represents Gaussian random noise, and the pressure $P=P(\rho)$ is expressed as an expansion centered around $\rho_m$, given by $P(\rho)=\sum_{n=1}^{\infty}\sigma_n(\rho-\rho_m)^n$, where $\rho_m$ denotes the mean of the local number density, and $\sigma_n$ stands for a constant in the pressure expansion. In comparison to the compressible Navier--Stokes equations, the TT model \eqref{original tt model} includes additional terms. Specifically, there exists a linear term $\alpha v$, responsible for injecting energy into the system, alongside a cubic nonlinear term $-\beta v|v|^2$, inducing a damping effect that leads to steady-states. Furthermore, $(v\cdot\nabla)^2 v$ acts akin to diffusion in terms of linear stability analysis.	 
	 
In the present paper, we focus on the deterministic scenario of \eqref{original tt model}, implying the absence of noise ($f=0$), with the ordered state ($\alpha>0$) and $P(\rho)=\rho-\rho_0$. Noticing that the values of parameters $\lambda$, $D_1$, $D_2$, and $D_L$ do not influence our analysis, we assume, without loss of generality, that $\lambda=D_1=D_2=D_L=1$. Consequently, the TT model \eqref{original tt model} simplifies to:
\begin{equation}\label{original tt model1}
	\begin{split}
	&v_t+v\cdot\nabla v=\alpha v-\beta \vert v\vert^2v+\nabla{\rm{ div}} v+\Delta v-\nabla \rho+(v\cdot\nabla)^2v, \\
	&\rho_t+{\rm{div}} (v \rho)=0.
	\end{split}
	\end{equation}
  	
Our primary aim is to investigate the stability of a steady-state solution of \eqref{original tt model1}, denoted by $(\rho_s, v_s)$, where both $\rho_s$ and $v_s$ are nonzero constants, with 
\bq\label{steady}
\rho_s=c \in \R_+ \quad \mbox{and} \quad v_s=\sqrt{\frac\alpha\beta}\bf{e}, 
\eq
 where $\bf{e}\in \mathbb{S}^{d-1}$ with $d=2,3$.  Essentially, if initial data are close to the aforementioned nonzero constant state, we demonstrate that the solutions of \eqref{original tt model1} converge to this steady state as time tends to infinity. More precisely, our first main result is stated as follows:	 
 
   \begin{theorem}\label{main result} The steady-state solution $(\rho_s, v_s)$ of the TT model \eqref{original tt model1} given by \eqref{steady} is stable in $H^m(\R^d)$-norm for any $m \geq 3$. Moreover, it is almost polynomially stable in $(\dot{H}^{-l}\cap H^m)(\R^d) \times(\dot{H}^{-l}\cap H^m)(\R^d)$ for some $l(m,d)>0$, in the sense that solutions which are initially close to $(\rho_s, v_s)$ converge to that in a norm $\dot{H}^s(\R^d) \times \dot{H}^s(\R^d)$ for all $-l < s \leq m$ at least polynomially fast as time tends to infinity.
\end{theorem}

 Following the initial introduction of the TT model, various derivatives have emerged. For instance, Gowrishankar et al. presented a model detailing the dynamics of self-propelled actin filaments and myosin, crucial for plasma membrane organization in two dimensions \cite{gowrishankar2012active}. Subsequently, Gowrishankar and Rao conducted numerical investigations into the collective motion of active filaments, as purportedly observed in experiments \cite{gowrishankar2016nonequilibrium}. The model proposed in \cite{gowrishankar2012active} is given by
 \begin{equation}\label{original simple model}
\begin{split}
&n_t+\lambda n\cdot\nabla n=\alpha n-\beta \vert n\vert^2n+K_1\nabla \ {\rm{div}} n+K_2\Delta n-\xi \nabla c+f,\\
&c_t+D\Delta c+\nu_0 \ {\rm{div}}\ (cn)=0.
\end{split}
\end{equation}
Here $c=c(x,t)$ represents the local concentration, $n=n(x,t)$ denotes the polarization vector at a given position, and $f$ signifies the athermal noise at position $x\in\mathbb{R}^d$ and time $t>0$. The parameters $\lambda$, $K_1$, $K_2$, $\xi$, $\alpha$, $\beta$, $\nu_0$, and $D_f$ are constants. Upon normalizing these parameters and modifying the notations, we rewrite the system \eqref{original simple model} as follows:
\begin{equation}\label{parabolic}
\begin{split}
&v_t+v\cdot\nabla v=\alpha v-\beta \vert v\vert^2v+\nabla{\rm{ div}} v+\Delta v-\nabla \rho, \\
&\rho_t+{\rm{div}}(v \rho)=\Delta \rho.
\end{split}
\end{equation}
We refer to \eqref{parabolic} as the Parabolic-Parabolic Toner--Tu (PPTT) model given the structural characteristic of the system. In contrast to the TT model \eqref{original tt model1}, the discrepancy in \eqref{parabolic} lies in the absence of $(v\cdot\nabla)^2v$ and the inclusion of diffusions for $\rho$, namely $\Delta\rho$.	
	
Our second main result is on the stability of steady-state solution $(\rho_s, v_s)$ of the system \eqref{parabolic}:
  \begin{theorem}\label{main simple} The steady-state solution $(\rho_s, v_s)$ of the PPTT model \eqref{parabolic} given by \eqref{steady} is stable in $H^m(\R^d)$-norm for any $m \geq 3$. Moreover, it is almost polynomially stable in $(\dot{H}^{-l}\cap H^m)(\R^d) \times (\dot{H}^{-l}\cap H^m)(\R^d)$ for some $l(m,d)>0$, in the same manner described in Theorem \ref{main result}.
  \end{theorem}
  
 \begin{remark} When $d=m=3$, the polynomial stability of the steady-state solution $(\rho_s, v_s)$ of the TT model \eqref{original tt model1} and PPTT model \eqref{parabolic} can be obtained in $H^3$, without introducing the negative Sobolev space, see Remark \ref{rmk_gd} for details.  
\end{remark}

To the best of our knowledge, there exists limited literature concerning the analytical study of the TT model. Specifically, for the incompressible TT model, only two research papers on the existence theory are available. Gibbon et al.  \cite{gibbon2023analytical} established the existence of Leray-type weak solutions for the three-dimensional case and global regularity for the two-dimensional case. More recently, Bae et al. \cite{bck} investigated the global existence and asymptotic stability of the regular solutions. However, the analytical analysis of the global-in-time existence of stability for the TT model \eqref{original tt model1} or PPTT model \eqref{parabolic} has not been explored yet. In this paper, for the first time, we established the global-in-time stability of solutions analytically. We would also like to emphasize that we prove stronger results than those stated in Theorems \ref{main result} and \ref{main simple}. Precisely, we do not impose the smallness assumptions on the initial data within the full Sobolev space $H^m(\R^d)$. Instead, we only need the smallness assumptions on $H^3$ and $H^{d-2}$ for the stability results concerning the TT model \eqref{original tt model1} and PPTT model \eqref{parabolic}, see Theorems \ref{simple result} and \ref{tt result} below.

\subsubsection*{Notations} All generic positive constants are denoted by $C$ which is independent of $t$, and $C(p_1, p_2, \dots )$ implies that a constant depends on $p_1, p_2, \dots$.  $f \ls g$ represents that there exists a positive constant $C>0$ such that $f \leq C g$. For simplicity, we denote $\int_{\mathbb{R}^d}\,dx$, $L^p(\R^d)$, and $H^m(\R^d)$ as  $\int \,dx$, $L^p$, and $H^m$, respectively.  We also denote $\|(u,\eta)\|^2_X$ as $\|u\|^2_X+\|\eta\|^2_X$, where $\|\cdot\|_X$ is the norm of the function space $X$. $[\cdot, \cdot]$ stands for the commutator operator, i.e. $[A,B] = AB - BA$. Finally, $\partial^k$ represents derivatives of order $k$ and $\Lambda^s$ stands for an operator whose Fourier multiplier is given by $|\xi|^s$

\subsubsection*{Outline of the paper} In the rest of this paper is organized as follows. In Section \ref{pre}, we furnish useful Sobolev inequalities and briefly outline the local well-posedness theory for the systems \eqref{original tt model1} and \eqref{parabolic} near the steady state $(\rho_s,  v_s)$. Section \ref{section simple} and Section \ref{globals tt} are dedicated to elucidating the proofs of Theorem \ref{main simple} and Theorem \ref{main result}, respectively.	
%
%
%
%
%
%
%
\section{Preliminaries}\label{pre}

%
%
%
%
%
%
%
\subsection{Useful inequalities}
In this subsection, we present technical lemmas that will be used to obtain sharp energy estimates of solutions later.

We first recall well-known inequalities whose proofs can be found in \cite{BCD11, grafakos2014modern, li2019kato}.
\begin{lemma}\label{big lemma} Let $s,s_1,s_2\in\mathbb{R}$, and $k\in\mathbb{N}$. Then the following holds.
 \begin{itemize}
 	\item[{\rm (i)}]  Let $u\in\dot{H}^{s_1}\cap\dot{H}^{s_2}$, then 
\[
	\|\partial^k u\|_{L^\infty}\ls \|u\|^{\frac{s_2-\left(k+\frac{d}{2}\right)}{s_2-s_1}} _{\dot{H}^{s_1}}\|u\|^{\frac{\left(k+\frac{d}{2}\right)-s_1}{s_2-s_1}}_{\dot{H}^{s_2}} \quad \text{for }  s_1<k+\frac{d}{2}<s_2.
\]
    \item[{\rm (ii)}]   Let $u\in\dot{H}^{s_1}\cap\dot{H}^{s_2}$, then
\[
	\| u\|_{\dot{H}^s}\leq \|u\|^{\frac{s_2-s}{s_2-s_1}} _{\dot{H}^{s_1}}\|u\|^{\frac{s-s_1}{s_2-s_1}}_{\dot{H}^{s_2}} \quad \text{for $s_1<s<s_2$.}
\]
\item[{\rm (iii)}]   Let $u,v\in L^\infty\cap\dot{H}^{s}$ with $s\in(0,\infty)$, then 
 	\begin{equation*}
	\|uv\|_{\dot{H}^s}\ls \|u\|_{L^\infty}\|v\|_{\dot{H}^s}+\|v\|_{L^\infty}\|u\|_{\dot{H}^s}. 
\end{equation*}
    \item[{\rm (iv)}]  Let $1<p<\infty$ with $1/p=1/p_1+1/q_1=1/p_2+1/q_2$, and $q_1,q_2\in (1,\infty)$, then
\[
	\|[\partial^k,u\cdot\nabla]v\|_{L^p}\ls \|\nabla u\|_{L^{p_1}}\|\partial^kv\|_{L^{q_1}}+\|\nabla v\|_{L^{p_2}}\|\partial^k u\|_{L^{q_2}}.
\]
   \item[{\rm (v)}]  Let $s\in (0,d)$, $1<p<q<\infty$ and $1/q+s/d=1/p$, then we have 
\[
	\|\Lambda^{-s} u\|_{L^q} \ls \|u\|_{L^p}.
\]
   \item[{\rm (vi)}]  Let $u\in \dot{H}^{\frac{d}{2}-s}$, then
\[
	\|u\|_{L^\frac{d}{s}}\ls \|\Lambda^{\frac{d}{2}-s} u\|_{L^2} \quad  \text{for } 0 <s<\frac{d}{2}.
\]
 \end{itemize}
\end{lemma}
 
We next provide several auxiliary estimates.

\begin{lemma}\label{lem_gd1}
 Let $k \in \N$ and $d=2$ or $d=3$. Then the following holds.
 \begin{itemize}
\item[{\rm (i)}] If $u, v, w \in H^{k+1}$ for $k>1$, then 
\[
	\|\nabla u\|_{L^\infty}\|\partial^kv\|_{L^2}\|\partial^kw\|_{L^2}  \ls (\| u\|_{H^\frac{(d-2)(k+1)}{2(k-1)}}+\|(v,w)\|_{L^2} ) \|\partial^{k+1} (u,v,w)\|^2_{L^2}.    
\]
\item[{\rm (ii)}] If $u,v \in H^{k+1}$ for $k\geq1$, then
\[
\begin{split}
	\| u\|_{L^\infty}\|\partial^kv\|_{L^2}& \ls (\| u\|_{H^\frac{(d-2)(k+1)}{2k}}+\|v\|_{L^2}) \|\partial^{k+1} (u,v)\|_{L^2}.
\end{split}
\]
\item[{\rm (iii)}] If $u,v \in H^{k+2}$ for $k\geq1$, then
\[
\begin{split}
	\| \nabla u\|_{L^\infty}\|\partial^kv\|_{L^2}& \ls (\| u\|_{H^\frac{(d-2)(k+2)}{2k}}+\|v\|_{L^2}) \|\partial^{k+2} (u,v)\|_{L^2}.
\end{split}
\]
\item[{\rm (iv)}] if $u \in \dot{H}^{\beta(d,k)} \cap \dot H^k$ and $v \in \dot{H}^{\beta(d,k)} \cap \dot H^{k+1}$ for $k>1$, then
\[
\begin{split}
	\| u\|_{L^\infty}\|\partial^kv\|^2_{L^2}& \ls \| (u,v)\|_{\dot{H}^{\beta(d,k)}}(\|\partial^k u\|^2_{L^2}+\|\partial^{k+1} v\|^2_{L^2}),
\end{split}
\]
where 
\[
\beta(d,k):=\frac{\left(\frac{d}{2}+k-1\right)-\sqrt{\left(\frac{d}{2}+k-1\right)^2+8k-2dk-2d}}{2}.
\]
\item[{\rm (v)}] If $u,v \in H^3 \cap H^{k+1}$ for $k>1$, then 
\[
	\|\partial^2 u\|_{L^{2d}}\|\partial^kv\|^2_{L^2}\ls  \|(u,v)\|_{H^3}  \|\partial^{k+1} (u,v)\|^2_{L^2}.
\]
\item[{\rm (vi)}] Let $d=3$. If $u,v,w \in H^3$, then 
\[
	\begin{split}
		\|\nabla u\|_{L^\infty}\|\partial^2v\|_{L^2}\|\partial^2w\|_{L^2}
	& \ls \|( u,v,w)\|_{H^1} \|\partial^3 (u,v,w)\|^2_{L^2}.
	\end{split}
\]
\end{itemize}    
\end{lemma}
\begin{proof}The proofs reply on the Agmon-type inequality, Lemma \ref{big lemma} (i) and the Sobolev interpolation inequality, Lemma \ref{big lemma} (ii). Noticing the condition on $k$, straightforward computations yield 
\begin{equation*}	
\begin{split}
 \|\nabla u\|_{L^\infty}\|\partial^kv\|_{L^2}\|\partial^kw\|_{L^2}&\ls \|\Lambda^\frac{(d-2)(k+1)}{2(k-1)} u\|^{\frac{k-1}{k+1}}_{L^2}\|\partial^{k+1} u\|^\frac{2}{k+1}_{L^2}\|v\|^\frac{1}{k+1}_{L^2}\|\partial^{k+1}v\|^\frac{k}{k+1}_{L^2}\|w\|^\frac{1}{k+1}_{L^2}\|\partial^{k+1}w\|^\frac{k}{k+1}_{L^2}\\
	& \ls(\| u\|_{H^\frac{(d-2)(k+1)}{2(k-1)}}+\|(v,w)\|_{L^2} ) \|\partial^{k+1} (u,v,w)\|^2_{L^2}, \cr
	\| u\|_{L^\infty}\|\partial^kv\|_{L^2} &\ls \|\Lambda^\frac{(d-2)(k+1)}{2k} u\|^{\frac{k}{k+1}}_{L^2}\|\partial^{k+1} u\|^\frac{1}{k+1}_{L^2}\|v\|^\frac{1}{k+1}_{L^2}\|\partial^{k+1}v\|^\frac{k}{k+1}_{L^2}\\
	& \ls (\| u\|_{H^\frac{(d-2)(k+1)}{2k}}+\|v\|_{L^2}) \|\partial^{k+1} (u,v)\|_{L^2} , \cr
		\|\nabla u\|_{L^\infty}\|\partial^kv\|_{L^2}&\ls \|\Lambda^\frac{(d-2)(k+2)}{2k} u\|^{\frac{k}{k+2}}_{L^2}\|\partial^{k+2} u\|^\frac{2}{k+2}_{L^2}\|v\|^\frac{2}{k+2}_{L^2}\|\partial^{k+2}v\|^\frac{k}{k+2}_{L^2}\\
	&\ls (\| u\|_{H^\frac{(d-2)(k+2)}{2k}}+\|v\|_{L^2})\|\partial^{k+2} (u,v)\|_{L^2},\cr
	\| u\|_{L^\infty}\|\partial^kv\|^2_{L^2} &\ls \|\Lambda^{\beta(d,k)} u\|^{\frac{k-\frac{d}{2}}{k-\beta(d,k)}}_{L^2}\|\partial^k u\|^\frac{\frac{d}{2}-\beta(d,k)}{k-\beta(d,k)}_{L^2}\|\Lambda^{\beta(d,k)}v\|^\frac{2}{k+1-\beta(d,k)}_{L^2}\|\partial^{k+1}v\|^\frac{2(k-\beta(d,k))}{k+1-\beta(d,k)}_{L^2}\\
	& \ls \| (u,v)\|_{\dot{H}^{\beta(d,k)}}(\|\partial^k u\|^2_{L^2}+\|\partial^{k+1} v\|^2_{L^2}), \cr
\|\partial^2 u\|_{L^{2d}}\|\partial^kv\|^2_{L^2}
	&\ls \|\Lambda^\frac{(d-1)(k+1)}{2(k-1)} u\|^{\frac{k-1}{k+1}}_{L^2}\|\partial^{k+1} u\|^\frac{2}{k+1}_{L^2}\|v\|^\frac{2}{k+1}_{L^2}\|\partial^{k+1}v\|^\frac{2k}{k+1}_{L^2}\\
	& \ls (\| u\|_{H^\frac{(d-1)(k+1)}{2(k-1)}}+\|v\|_{L^2})  \|\partial^{k+1} (u,v)\|^2_{L^2} \cr
	&\ls  \|(u,v)\|_{H^3}  \|\partial^{k+1} (u,v)\|^2_{L^2}, 
\end{split}
\end{equation*}
and
\begin{equation*}
	\begin{split}
		\|\nabla u\|_{L^\infty}\|\partial^2v\|_{L^2}\|\partial^2w\|_{L^2}&\ls \| u\|^{\frac{1}{6}}_{L^2}\|\partial^3 u\|^\frac{5}{6}_{L^2}\|v\|^\frac{5}{12}_{H^1}\|\partial^3v\|^\frac{7}{12}_{L^2}\|w\|^\frac{5}{12}_{H^1}\|\partial^3w\|^\frac{7}{12}_{L^2} \ls \|( u,v,w)\|_{H^1}\|\partial^3 (u,v,w)\|^2_{L^2}.
	\end{split}
\end{equation*}
This completes the proof.
\end{proof}

%
%
%
%
%
%
%
\subsection{Reformulation}\label{sec_refo}
In order to study the stability of steady state $(\rho_s, v_s)$, we seek solutions $\rho$ and $v$ of the form $\rho=\eta+\rho_s$ and $v = u+v_s$. For simplicity, we fix $\alpha=\beta=1$.
Given that our analysis remains invariant under the magnitude of $\rho_s$ and the rotation of $v_s$, we set $\rho_s=1$ and $v_s={\rm e}_1$. Consequently,  the system \eqref{original tt model1} can be reformulated as follows:  
	\begin{equation}\label{original}
		\begin{split}
			&u_t+u\cdot\nabla u-\nabla {\rm div}\ u-\Delta u+\nabla \eta=-2\bar{u}{{\rm e}_1}-\vert u\vert^2u-{{\rm e}_1}|u|^2-2\bar{u}u-{{\rm e}_1}\cdot\nabla u\\
			&\hspace{5.5cm} +(u+{{\rm e}_1})\cdot\nabla ((u+{{\rm e}_1})\cdot \nabla u),\\
	&\eta_t=-\nabla \cdot (\eta u)-{{\rm e}_1}\cdot \nabla \eta -\nabla\cdot u, \quad \bar u={{\rm e}_1}\cdot u
	\end{split}
	\end{equation}
	with initial data
\[
(u,\eta)(x,0) = (u_0,\eta_0)(x), \quad x \in \R^d.
\]

Similarly, we also consider the pair $(\eta,u)$ satisfying the following perturbed PPTT equations: 
\begin{equation}\label{simple model}
\begin{split}
&u_t+u\cdot\nabla u-\nabla {\rm div} u-\Delta u+\nabla \rho=-2\bar{u}{{\rm e}_1}-\vert u\vert^2u-|u|^2{\rm e_1}-2\bar{u}u-{{\rm e}_1}\cdot\nabla u,\\
&\eta_t-\Delta \eta=-\nabla \cdot (\eta u)-{{\rm e}_1}\cdot \nabla \eta -\nabla\cdot u.
\end{split}
\end{equation}

%
%
%
%
%
%
%

\subsection{Local well-posedness}\label{local}

In this part, we present the local-in-time existence and uniqueness of regular solutions to the systems \eqref{original} and \eqref{simple model}. Since the local existence theory for similar hydrodynamic models has been well-developed in the $H^s$ Sobolev space, we state the theorem below without providing its detailed proof. We refer to \cite{danchin2015fourier,Maj03, majda2002vorticity} and references therein for the readers who are interested in it.

\begin{theorem}\label{thm_lwp} Let $m \geq 3$ and $d=2$ or $d=3$. Suppose that $(u_0, \eta_0) \in H^m \times H^m$. Then, for any positive constants $N < M$, there exists a positive cosntant $T_0$ depending only on $N$ and $M$ such that if $\|(u_0, \eta_0)\|_{H^m} \leq N$, then the systems \eqref{original} and \eqref{simple model} admits a unique solution $(u,\eta) \in C([0,T_0];H^m) \times C([0,T_0];H^m)$ satisfying
\[
\sup_{0 \leq t \leq T_0} \|(u,\eta)(t)\|_{H^m} \leq M.
\]
\end{theorem}
\begin{remark}Due to the structures of the systems, we additionally have $u \in L^2(0,T;H^{m+1})$ and $u, \eta \in L^2(0,T;H^{m+1})$ for the systems  \eqref{original} and \eqref{simple model}, respectively. 
\end{remark}

%
%
%
%
%
%
%

\section{Cauchy problem for the Parabolic-Parabolic Toner--Tu model}\label{section simple}

In this section, we study the global-in-time existence of regular solutions to the system \eqref{parabolic} near the steady state $(\rho_s, v_s)$ and its large-time behavior, thereby establishing the proof of Theorem \ref{main simple}.
 
%
%
%
%
%
%
%
\subsection{Global-in-time existence of solutions}

In this part, our main purpose is to prove the following theorem on the global existence and uniqueness of $H^m$ solutions to the system \eqref{simple model}. 

\begin{theorem}\label{simple result}
	Let $m\geq 3$ be an integer with $d=2,3$. If $(u_0,\eta_0)\in H^m$ with $\|(u_0,\eta_0)\|_{H^{d-2}}\leq \epsilon_0$ for sufficiently small $\epsilon_0>0$, then there exists a unique solution $(u,\eta)$ of the PPTT model \eqref{simple model} satisfying
\begin{equation*}
u\in C([0,\infty);H^m)\cap L^2(0,\infty; \dot{H}^{m+1}) \quad \text{and} \quad  \eta\in C([0,\infty);H^m) \cap L^2(0,\infty; \dot{H}^{m+1}).
	\end{equation*}
	Moreover, we have 
    \begin{equation*}
		\|( u, \eta)(t)\|^2_{H^m}+\int^t_0(\|\nabla(u, \eta)(s)\|^2_{H^m}+\|\bar{u}(s)\|^2_{H^m})\,ds \leq  \|( u_0, \eta_0)\|^2_{H^m} \quad \text{for all} \ \ t\geq0. 	
\end{equation*}
\end{theorem}

\begin{remark}Theorem \ref{simple result} shows that the steady-state solution $(1, {\rm e}_1)$ is stable in $H^m$ for any $m \geq 3$. 
\end{remark}

\begin{remark} To construct the global existence of a solution in $H^m \times H^m$, we only impose the smallness assumption on the initial data in $H^{d-2}$. That is, in the two-dimensional case, we have a global regular solution only under the smallness assumption on the initial data in $L^2$.
\end{remark}

In the rest of this subsection, we focus on the a priori estimate of solutions $(u,\eta)$ in the function space $C([0,T]; H^m) \times C([0,T]; H^m)$. Then this together with the standard continuation argument yields that the local-in-time unique regular solution mentioned in Section \ref{local} can be extended to the global one. For notational simplicity, we set
\[
X_m(T) := C([0,T]; H^m) \cap L^2(0,T; H^{m+1}). 
\]

\begin{lemma}\label{priori for parabolic}  Let $T>0$, $m\geq 3$, and $(u,\eta) \in X_m(T) \times X_m(T)$ be a solution to the system  \eqref{simple model}. Then there exists $C=C(k,d)>0$ such that 
\begin{align}\label{energy_PPTT}
\begin{aligned}
	&\frac{1}{2}\frac{d}{dt}\|\partial^k(u,\eta)\|^2_{L^2}+ \frac34 \|\partial^k\nabla(u,\eta)\|^2_{L^2}+\|\partial^k\bar{u}\|^2_{L^2}\\
	&\quad \leq C\left(\|(u,\eta)\|_{H^{d-2}}+\|(u,\eta)\|^2_{H^{d-2}}\right)\|\partial^k\nabla (u, \eta)\|^2_{L^2}+\frac{1}{32}\chi_{k\geq1} \|\partial^k u\|^2_{L^2}
    \end{aligned}
\end{align}
for $k=0,\dots, m$, where $\chi$ denotes the characteristic function.
\end{lemma}
\begin{proof} 
To obtain the term $\|(u,\eta)\|^2_{H^{d-2}}$ appropriately, it is crucial to have sharp estimates of the lower-order derivatives of solutions. Thus, we first present the $\dot H^k$ estimates with $k=0,1$ and then address the higher-order estimates.

\medskip

$\bullet$ ($\dot H^k$-estimate with $k=0,1$) We begin with the $L^2$ estimates. By using Lemma \ref{big lemma} (ii) and (vi), we deduce
\begin{equation}\label{L2 esimate for sm}
	\begin{split}
		&\frac{1}{2}\frac{d}{dt}\|(u,\eta)\|^2_{L^2}+\|\nabla ( u,\eta)\|^2_{L^2}+\|\nabla\cdot u\|^2_{L^2}+2\|\bar{u}\|^2_{L^2}+\|u\|^4_{L^4}\\
		&\quad =-\int (u\cdot\nabla u)\cdot u \,dx -3\int \bar{u}|u|^2\,dx-\int\eta\nabla\cdot(\eta u) \,dx\\
		&\quad \leq  \|\nabla u\|_{L^2}\|u\|^2_{L^4}+3\|\bar{u}\|_{L^2}\|u\|_{L^3}\|u\|_{L^6}+\|\nabla\eta\|_{L^2}\|\eta\|_{L^6}\| u\|_{L^3}\\
		&\quad \ls \|\nabla u\|^2_{L^2}\|u\|_{H^{d-2}}+ \|\bar{u}\|_{L^2}\|\Lambda^\frac{d}{6} u\|_{L^2}\|\Lambda^\frac{d}{3} u\|_{L^2}+ \|\nabla\eta\|_{L^2}\|\Lambda^\frac{d}{3} \eta\|_{L^2}\|\Lambda^\frac{d}{6} u\|_{L^2}\\
		&\quad \ls \|\nabla u\|^2_{L^2}\|u\|_{H^{d-2}}+ \|\bar{u}\|_{L^2}\|\Lambda^{\frac{3}{2}-\frac{3}{d}} u\|^\frac{d}{3} _{L^2}\|\nabla u\|^{1-\frac{d}{3}} _{L^2}\|u\|^{1-\frac{d}{3}} _{L^2}\|\nabla u\|^\frac{d}{3} _{L^2}\\
		&\qquad + \|\nabla\eta\|_{L^2}\|\Lambda^{\frac{3}{2}-\frac{3}{d}} u\|^\frac{d}{3} _{L^2}\|\nabla u\|^{1-\frac{d}{3}} _{L^2}\|\eta\|^{1-\frac{d}{3}} _{L^2}\|\nabla \eta\|^\frac{d}{3} _{L^2}\\
		&\quad \leq C \|(u,\eta)\|_{H^{d-2}}\|\nabla (u,\eta)\|^2_{L^2} + \|\bar u\|_{L^2}^2.
	\end{split}
\end{equation}
Similarly, we also find
\begin{equation*}
	\begin{split}
		\frac{1}{2}&\frac{d}{dt}\|\partial (u,\eta)\|^2_{L^2}+\|\partial\nabla (u,\eta)\|^2_{L^2}+\|\partial\nabla \cdot u\|^2_{L^2}+2\|\partial   \bar{u}\|^2_{L^2}\\
		&= \int (u\cdot\nabla u)\cdot\partial^2 u \,dx+\int |u|^2u\cdot\partial^2 u \,dx+\int |u|^2\cdot\partial^2 \bar{u} \,dx+2\int (\bar{u}u)\cdot \partial^2 u \,dx+\int  \partial^2\eta\nabla \cdot(\eta u) \,dx\\
		&=:I+II+III+IV+V,
 	\end{split}
\end{equation*}
where we estimate
\begin{equation*}
	\begin{split}
	I&\leq  \| u\|_{L^3}\|\nabla u\|_{L^6}\|\partial^2 u\|_{L^2}\ls \|\Lambda^\frac{d}{6} u\|_{L^2}\|\nabla\Lambda^\frac{d}{3} u\|_{L^2}\|\partial^2 u\|_{L^2}\\
& \ls \|\Lambda^\frac{3d-6}{3+d} u\|^{\frac{1}{2}+\frac{d}{6} }_{L^2}\|\partial^2 u\|^{\frac{1}{2}-\frac{d}{6}} _{L^2}\| u\|^{\frac{1}{2}-\frac{d}{6} }_{L^2}	\|\partial^2 u\|^{\frac{3}{2}+\frac{d}{6}}_{L^2}\ls \| u\|_{H^{d-2}}\|\partial^2 u\|^2_{L^2}, \cr
	II&=-\int \partial(|u|^2u)\cdot\partial u\,dx\ls \|u\|^2_{L^\infty}\|\partial u\|^2_{L^2}\ls \|u\|^2_{H^{d-2}}\|\nabla\partial u\|^2_{L^2}, \cr
		III+IV
		&\ls \|u\|_{L^3}\|\nabla u\|_{L^6}\|\nabla u\|_{L^2}  \ls  \|\Lambda^\frac{3d-6}{3+d} u\|^{\frac{1}{2}+\frac{d}{6} }_{L^2}	\|\partial^2 u\|^{\frac{1}{2}-\frac{d}{6}} _{L^2}\| u\|^{\frac{1}{2}-\frac{d}{6} }_{L^2}	\|\partial^2 u\|^{\frac{1}{2}+\frac{d}{6}}_{L^2}\|\nabla u\|_{L^2}\\
		& \ls \|u\|_{H^\frac{3d-6}{3+d}}\|\partial^2 u\|_{L^2}\|\nabla u\|_{L^2} \ls \|u\|_{H^{d-2}}\|\nabla u\|^2_{H^1}, 
	\end{split}
\end{equation*}
and
\begin{equation*}
	\begin{split}
		V& =\int (\nabla\eta \cdot u)\partial^2 \eta \,dx+\int (\eta\nabla \cdot u)\partial^2 \eta \,dx\\
		& \leq   \|u\|_{L^3}\|\nabla \eta\|_{L^6}\|\partial^2\eta\|_{L^2} + \|\nabla u\|_{L^6}\| \eta\|_{L^3}\|\partial^2\eta\|_{L^2}\\
		& \ls \|\Lambda^\frac{3d-6}{3+d} u\|^{\frac{1}{2}+\frac{d}{6} }_{L^2}	\|\partial^2 u\|^{\frac{1}{2}-\frac{d}{6}} _{L^2}\| \eta\|^{\frac{1}{2}-\frac{d}{6} }_{L^2}	\|\partial^2 \eta\|^{\frac{3}{2}+\frac{d}{6}}_{L^2} + \|\Lambda^\frac{3d-6}{3+d} \eta\|^{\frac{1}{2}+\frac{d}{6} }_{L^2}	\|\partial^2 \eta\|^{\frac{3}{2}-\frac{d}{6}} _{L^2}\| u\|^{\frac{1}{2}-\frac{d}{6} }_{L^2}	\|\partial^2 u\|^{\frac{1}{2}+\frac{d}{6}}_{L^2}\\
		& \ls \| (u, \eta)\| _{H^{d-2}}\left(\|\partial^2\eta\|^2_{L^2}+ \|\partial^2 u\|^2 _{L^2}\right). 
	\end{split}
\end{equation*}
due to Lemma \ref{big lemma} (ii) and (vi). This yields 
\begin{equation}\label{PPTT_H1}
		\frac{1}{2}\frac{d}{dt} \|\pa^k (u, \eta)\|^2_{L^2}+\|\pa^k \nabla  (u, \eta)\|^2_{L^2}+\|\bar{u}\|^2_{L^2}
		\leq C\|(u,\eta)\|_{H^{d-2}} \|\pa^k \nabla  (u,\eta)\|^2_{L^2}   +\frac{1}{32}\chi_{k\geq1} \|\pa^k u\|^2_{L^2}
\end{equation}
for $k=0,1$.

\medskip

$\bullet$ ($H^k$-estimate with $k\geq 2$) For $k=2,\dots, m$, we get\\
\begin{equation}\label{tt priori estimate}
	\begin{split}
		&\frac{1}{2}\frac{d}{dt}\|\partial^k (u, \eta)\|^2_{L^2}+\|\partial^k\nabla (u,\eta)\|^2_{L^2}+\|\partial^k\nabla\cdot u\|^2_{L^2}+2\|\partial^k   \bar{u}\|^2_{L^2}\\
		&\quad =-\int \partial^k(u \cdot\nabla u)\partial^k u \,dx-\int \partial^k(u |u|^2)\cdot\partial^k u \,dx-\int \partial^k( |u|^2)\partial^k\bar{u} \,dx\\
		&\qquad -2\int \partial^k(\bar{u} u)\cdot\partial^k u \,dx- \int \partial^k(\nabla\cdot(\eta  u))\partial^k \eta \,dx\\
		&\quad =: I+II+III+IV+V.
		\end{split}
\end{equation}
By Lemma \ref{big lemma} (iv), Lemma \ref{lem_gd1} (i), (ii), and (vi), we obtain
\begin{equation*}
\begin{split}
		I  =-\int \left([\partial^k, u\cdot\nabla]u\right)\cdot \partial^k u \,dx-\int \left(u\cdot\partial^k\nabla u\right) \cdot\partial^k u \,dx\ls \|\nabla u\|_{L^\infty}\|\partial^k u\|^2_{L^2}
	 \ls \|u\|_{H^{d-2}}\|\partial^{k+1}u\|_{L^2}^2,
\end{split}
\end{equation*}
\begin{equation*}
\begin{split}
	II& \ls (\|u\|_{L^\infty}\|\partial^k(|u|^2)\|_{L^2}+\||u|^2\|_{L^\infty}\|\partial^ku\|_{L^2})\|\partial^k u\|_{L^2} \ls \|u\|^2_{L^\infty}\|\partial^ku\|^2_{L^2}\ls \|u\|^2_{H^{d-2}}\|\partial^{k+1}u\|^2_{L^2},
\end{split}
\end{equation*}
\begin{equation*}
\begin{split}
	III& \ls \|u\|_{L^\infty}\|\partial^ku\|_{L^2}\|\partial^k\bar{u}\|_{L^2} \leq C\|u\|^2_{L^\infty}\|\partial^ku\|^2_{L^2}+\frac{1}{32}\|\partial^k\bar{u}\|^2_{L^2}   \leq C\|u\|^2_{H^{d-2}}\|\partial^{k+1}u\|^2_{L^2}+\frac{1}{32}\|\partial^k\bar{u}\|^2_{L^2},
	\end{split}
\end{equation*}
\begin{equation*}
\begin{split}
	IV&\ls (\|\bar{u}\|_{L^\infty}\|\partial^ku\| _{L^2}+\|u\|_{L^\infty}\|\partial^k\bar{u}\| _{L^2})\|\partial^ku\| _{L^2}\\
	& \ls (\|\bar{u}\|_{L^\infty}\|\partial^ku\| _{L^2}+\|u\|_{L^\infty}\|\partial^k\bar{u}\| _{L^2})\|\partial^ku\| _{L^2}\\
	& \leq  C\|u\|^2_{L^\infty}\|\partial^k u\|^2 _{L^2}+\frac{1}{32}\|\partial^k u\|^2_{L^2}\\
	& \leq  C\|u\|^2_{H^{d-2}}\|\partial^{k+1} u\|^2_{L^2}+\frac{1}{32} \|\partial^ku\|^2_{L^2},
	\end{split}
\end{equation*}
and
\begin{equation*}\label{temporal 4}
\begin{split}
		V&=- \int \partial^k\nabla\cdot (\eta  u)\partial^k \eta \,dx=-\int \partial^k\eta (\partial^k(\nabla\eta\cdot u) \,dx+\partial^k(\eta\nabla\cdot u)) \,dx\\
	& = -\int \partial^k\eta [\partial^k,u\cdot\nabla]\eta \,dx-\int \partial^k\eta [\partial^k,\eta\nabla\cdot]u \,dx-\int \partial^k\eta (u\cdot\nabla \partial^k\eta)\,dx-\int \partial^k\eta (\eta\partial^k\nabla\cdot u) \,dx\\
	& \leq C\left(\|\nabla u \|_{L^\infty}\|\partial^k\eta\|_{L^2}+\|\nabla \eta\|_{L^\infty}\|\partial^k u\|_{L^2}\right)\|\partial^k \eta\|_{L^2}\\
	&\quad +C\|\nabla u \|_{L^\infty}\|\partial^k\eta\|^2_{L^2}+C\|\eta\|^2_{L^\infty}\|\partial^k \eta\|^2_{L^2}+\frac{1}{32}\|\partial^k\nabla u\|^2_{L^2}	\\
	& \leq   C\left(\|(u,\eta)\|_{H^{d-2}}+\|(u,\eta)\|^2_{H^{d-2}}\right)\|\partial^k\nabla (u, \eta)\|^2_{L^2}+\frac{1}{32}\|\partial^k\nabla u\|^2_{L^2}.
	\end{split}
\end{equation*}
Combining all of the above estimates deduces
\begin{equation*}
    \begin{split}
	\frac{1}{2}&\frac{d}{dt}\|\partial^k(u,\eta)\|^2_{L^2}+\|\partial^k\nabla(u,\eta)\|^2_{L^2}+2\|\partial^k\bar{u}\|^2_{L^2}\\
	&\quad \leq  C\left(\|(u,\eta)\|_{H^{d-2}}+\|(u,\eta)\|^2_{H^{d-2}}\right)\|\partial^k\nabla (u, \eta)\|^2_{L^2}+\frac{1}{32}\left(\|\partial^k u\|^2_{H^1}+\|\partial^k \bar{u}\|^2_{L^2}\right)
    \end{split}
\end{equation*}
for $k\geq 2$. This together with \eqref{PPTT_H1} concludes the desired result.
\end{proof}

We now provide the details of the proof for the global existence of regular solutions.

\begin{proof}[Proof of Theorem \ref{simple result} ]
By taking $\e_0 > 0$ small enough, we obtain from Lemma \ref{priori for parabolic} that
\[
\frac{1}{2}\frac{d}{dt}\|\partial^k(u,\eta)\|^2_{L^2}+\frac34\|\partial^k\nabla(u,\eta)\|^2_{L^2}+\|\partial^k\bar{u}\|^2_{L^2} \leq \frac{1}{32}\chi_{k\geq1}\|\partial^k u\|^2_{L^2}.
\]
Then, we sum over $k$ to have
\begin{equation*}
 \frac{d}{dt}\|(u,\eta)\|^2_{H^m}+ \|\nabla( u, \eta)\|^2_{H^m}+ \|\bar{u}\|^2_{H^m}\leq 0,
\end{equation*}
and subsequently, we conclude
\begin{equation*}
	\begin{split}
		\|(u,\eta)(t)\|^2_{H^m}+\int^t_0 \left(\|\nabla(u, \eta)(s)\|^2_{H^m}+\|\bar{u}(s)\|^2_{H^m}\right)ds& \leq  \|( u_0,\eta_0)\|^2_{H^m} \quad \text{for all} \ \ t\geq0. 
		\end{split}
\end{equation*}
This completes the proof.
\end{proof}

%
%
%
%
%
%
%
\subsection{Large-time behavior}\label{sec_lt_pptt}
In this part, we provide the large-time behavior of solutions $(u,\eta)$ to the system \eqref{simple model} showing polynomial decay rates of convergence of solutions. In particular, this result shows that the steady-state solution $(1, {\rm e}_1)$ is almost polynomially stable.

As observed in the previous section, Lemma \ref{priori for parabolic} provides the uniform-in-time bound of the solutions $(u,\eta)$ in $H^m \times H^m$, however, it does not give an appropriate dissipation rate for the large-time behavior of solutions with a decay rate. In a nutshell, to absorb the term $\chi_{k \geq 1}\|\pa^k u\|_{L^2}^2$ in the energy estimate \eqref{energy_PPTT}, we need to use the dissipation rate from the estimate of $\|(u,\eta)\|_{L^2}^2$, but it is difficult to obtain the zeroth-order dissipation rate from $\|\pa^k \nabla(u,\eta)\|_{L^2}$ for any $k \geq 1$. In order to overcome that difficulty, inspired by \cite{guo2012decay}, we take into account the negative Sobolev space for our solutions.

\begin{lemma}\label{lem_ns_pptt}Let $0< s<\frac{d}{2}$ and the assumptions of Lemma \ref{priori for parabolic} be satisfied. We further assume that $u,\eta \in C([0,T]; \dot H^{-s}) \cap L^2(0,T; \dot H^{1-s})$. Then we have
\begin{equation*}
	\begin{split}
		&\frac{1}{2}\frac{d}{dt}\|  (u,\eta)\|^2_{\dot{H}^{-s}} + \| \nabla (u,\eta)\|^2_{\dot{H}^{-s}}+\|\bar{u}\|^2_{\dot{H}^{-s}}  \cr
		&\quad \leq C(\| (u,\eta)\|_{H^{d-2}\cap\dot{H}^{-s}}+\| (u,\eta)\|^2_{H^{d-2}\cap\dot{H}^{-s}}) \|\nabla (u,\eta)\|^2_{H^{d-2}\cap\dot{H}^{-s}} + \frac12 \|\bar u\|_{L^2}^2 .
	\end{split}
\end{equation*}
Moreover, we obtain
\begin{align*}
		&\frac{1}{2}\frac{d}{dt}\|  (u,\eta)\|^2_{H^{d-2}\cap\dot{H}^{-s}} + \| \nabla (u,\eta)\|^2_{H^{d-2}\cap\dot{H}^{-s}}+\|\bar{u}\|^2_{\dot{H}^{-s}} + \frac12\|\bar u\|_{H^{d-2}}^2  \cr
		&\quad \leq C(\| (u,\eta)\|_{H^{d-2}\cap\dot{H}^{-s}}+\| (u,\eta)\|^2_{H^{d-2}\cap\dot{H}^{-s}}) \|\nabla (u,\eta)\|^2_{H^{d-2}\cap\dot{H}^{-s}}.
\end{align*}
\end{lemma}
\begin{proof}
A straightforward computation gives
\begin{equation*}
    \begin{split}
	&	\frac{1}{2}\frac{d}{dt}\|  (u,\eta)\|^2_{\dot{H}^{-s}} + \| \nabla (u,\eta)\|^2_{\dot{H}^{-s}}+\| \nabla\cdot u\|^2_{\dot{H}^{-s}} + 2\|\bar u\|_{\dot H^{-s}}^2\\
	&\quad =  -\int \Lambda^{-s}\nabla\cdot(\eta u)\Lambda^{-s}\eta \,dx-\int \Lambda^{-s}(u\cdot\nabla u)\Lambda^{-s}u \,dx-\int \Lambda^{-s}(|u|^2 u)\Lambda^{-s}u \,dx\\
	&\qquad -\int \Lambda^{-s}(| u|^2)\Lambda^{-s}\bar{u}\,dx -2\int \Lambda^{-s}(\bar u u)\Lambda^{-s}u \,dx\\
	&\quad =: I+II+III+IV+V.
    \end{split}
\end{equation*}
Note that 
\[
\|\nabla\cdot(\eta u)\|_{\dot H^{-s}} 
\ls \|u\cdot\nabla\eta\|_{L^{\frac{1}{\frac{1}{2}+\frac{s}{d}}}} + \|\eta\nabla\cdot u\|_{L^{\frac{1}{\frac{1}{2}+\frac{s}{d}}}} \ls \| u \|_{\dot H^{\frac d2-s}}\|\nabla \eta\|_{L^2} + \| \eta \|_{\dot H^{\frac d2-s}}\|\nabla u\|_{L^2} 
\]
thanks to Lemma \ref{big lemma} (v) and (vi). Similarly, we get
\[
 \|\Lambda^{-s}(u\cdot\nabla u)\|_{L^2} \ls \| u \|_{\dot H^{\frac d2-s}}\|\nabla u\|_{L^2} \quad \mbox{and} \quad \| \Lambda^{-s}| u|^2\|_{L^2} \ls \| u \|_{\dot H^{\frac d2-s}}\|u\|_{L^2}.
\]
Using those estimates, we obtain
\begin{align*}
I + II  + V&\ls \lt(\| u \|_{\dot H^{\frac d2-s}}\|\nabla \eta\|_{L^2} + \| \eta \|_{\dot H^{\frac d2-s}}\|\nabla u\|_{L^2} \rt)\|\eta\|_{\dot H^{-s}} +  \| u \|_{\dot H^{\frac d2-s}}\|\nabla u\|_{L^2}\|u\|_{\dot H^{-s}}  \cr
&\quad  +\| u \|_{\dot H^{\frac d2-s}}\|\bar u\|_{L^2}\|u\|_{\dot H^{-s}}\cr
&\leq C\|\nabla (u,\eta)\|_{H^{d-2}\cap\dot{H}^{-s}}^2\|(u,\eta)\|_{\dot H^{-s}}(1 + \|(u,\eta)\|_{\dot H^{-s}}) + \frac12 \|\bar u\|_{L^2}^2,
\end{align*}
and
\[
IV \ls \|\Lambda^{\frac d2-s} u\|_{L^2}\|u\|_{L^2}\|\Lambda^{-s} \bar u\|_{L^2}  \leq C\| u \|_{\dot H^{\frac d2-s}}^2\|u\|_{L^2}^2 + \frac12 \|\bar u\|_{\dot H^{-s}}^2 \leq C\|\nabla u\|_{H^{d-2}\cap\dot{H}^{-s}}^2\|u\|_{L^2}^2 + \frac12 \|\bar u\|_{\dot H^{-s}}^2 ,
\]
where we used $\| u \|_{\dot H^{\frac d2-s}} \leq \|\nabla u\|_{H^{d-2}\cap\dot{H}^{-s}}$ due to $1 - s \leq \frac d2 -s \leq d-1$.

For $III$, we use the estimate in \eqref{L2 esimate for sm} to get
\[
\||u|^2\|_{L^2} \ls \|u\|_{L^3}\|u\|_{L^6} \ls \|u\|_{H^{\frac32 - \frac3d}} \|\nabla u\|_{L^2},
\]
and thus
\begin{align*}
III &\ls \||u|^2 u\|_{L^{\frac{1}{\frac{1}{2}+\frac{s}{d}}}}\|\Lambda^{-s}u\|_{L^2} \cr
&\ls \| u \|_{\dot H^{\frac d2-s}}  \|| u|^2\|_{L^2}\|u\|_{\dot H^{-s}} \cr
& \ls \| u \|_{\dot H^{\frac d2-s}}  \|u\|_{H^{\frac32 - \frac3d}} \|\nabla u\|_{L^2}\|u\|_{\dot H^{-s}}\cr
&\ls \| u\|^2_{H^{d-2}\cap\dot{H}^{-s}} \|\nabla u\|^2_{H^{d-2}\cap\dot{H}^{-s}}.
\end{align*}
We then combine all of the above estimates to conclude the first assertion. 

The second assertion readily follows from the combination of the above and Lemma \ref{priori for parabolic}.
\end{proof}


 Under the additional smallness assumptions on $\|(u,\eta)\|_{\dot H^{-s}}$ for $0 < s < \frac d2$, Lemma \ref{lem_ns_pptt} gives the uniform-in-time boundedness of $\|(u,\eta)\|_{\dot H^{-s}}$. Together with this, we modify the $H^m$ estimates of solutions \eqref{energy_PPTT} to get the temporal decay rates of the PPTT model \eqref{simple model} in the lemma below. In brief, we remove the $\| u\|_{\dot H^1}^2$ term on the right-hand side of \eqref{energy_PPTT} that allows us to employ the Gagliardo--Nirenberg interpolation inequality Lemma \ref{big lemma} (ii) to get appropriate dissipation rates.
 
\begin{lemma}\label{decay estimate for parabolic} Let $T>0$, $m\geq 3$, $\max\{0, -\beta(d,m)\} \leq s<\frac{d}{2}$ with 
\bq\label{betadm}
\beta(d,m)=\frac{\left(\frac{d}{2}+m-1\right)-\sqrt{\left(\frac{d}{2}+m-1\right)^2+8m-2dm-2d}}{2} \leq 0, 
\eq
and the assumptions of Lemma \ref{lem_ns_pptt} be satisfied.  
Then we have
\begin{equation*}
\begin{split}
	&\frac{1}{2}\frac{d}{dt}\|\partial^k (u, \eta)\|^2_{L^2}+\|\partial^k\nabla (u,\eta)\|^2_{L^2}+2\|\partial^k   \bar{u}\|^2_{L^2}\\
&\quad \leq C(\| (u,\eta)\|_{{H^{d-2}}\cap \dot{H}^{\beta(d,k)}}+\|(u,\eta)\|^2_{H^{d-2}})(\|\partial^k\nabla (u,\eta)\|^2_{L^2}+\|\partial^k \bar{u}\|^2_{L^2})\\
&\qquad +\frac{1}{32}\left(\|\partial^k\nabla u\|^2_{L^2}+\|\partial^k \bar{u}\|^2_{L^2}\right)
\end{split}
\end{equation*}
for $k=2,\dots, m$, where $C > 0$ is independent of $t$.
\end{lemma}
\begin{proof}
We first recall the inequality \eqref{tt priori estimate}:
\begin{equation*}
	\begin{split}
		\frac{1}{2}&\frac{d}{dt}\|\partial^k (u, \eta)\|^2_{L^2}+\|\partial^k\nabla (u,\eta)\|^2_{L^2}+\|\nabla\cdot\partial^k u\|^2_{L^2}+2\|\partial^k   \bar{u}\|^2_{L^2}\\
		&=-\int \partial^k(u \cdot\nabla u)\partial^k u \,dx-\int \partial^k(u |u|^2)\partial^k u \,dx-\int \partial^k( |u|^2)\partial^k \bar{u} \,dx\\
		&\quad -2\int \partial^k(\bar{u} u)\partial^k u \,dx- \int \partial^k\nabla\cdot(\eta  u)\partial^k \eta \,dx\\
		&=:I+II+III+IV+V,
		\end{split}
\end{equation*}
where
\begin{equation*}
\begin{split}
	I+II+III+V\leq C\left(\|(u,\eta)\|_{H^{d-2}}+\|(u,\eta)\|^2_{H^{d-2}}\right)\|\partial^k\nabla (u, \eta)\|^2_{L^2}+\frac{1}{32}\left(\|\partial^k\nabla u\|^2_{L^2}+\|\partial^k \bar{u}\|^2_{L^2}\right).
\end{split}
\end{equation*}
For $IV$, we use Lemma \ref{big lemma} (iii) and Lemma \ref{lem_gd1} (ii), (iv) to deduce
\begin{equation*}
\begin{split}
	IV&=-2\int \partial^k(\bar{u} u)\partial^k u\leq C(\|\bar{u}\|_{L^\infty}\|\partial^ku\| _{L^2}+\|u\|_{L^\infty}\|\partial^k\bar{u}\| _{L^2})\|\partial^ku\| _{L^2}\\
	& \leq  C(\|\bar{u}\|_{L^\infty}\|\partial^ku\| _{L^2}+\|u\|_{L^\infty}\|\partial^k\bar{u}\| _{L^2})\|\partial^ku\| _{L^2}\\
	& \leq  C\|u\|_{\dot{H}^{\beta(d,k)}}\left(\|\partial^k \bar{u}\|^2_{L^2}+\|\partial^{k+1}u\|^2_{L^2}\right)+C\|u\|_{H^\frac{(d-2)(k+1)}{2k}}\left(\|\partial^k \bar{u}\|^2_{L^2}+\|\partial^{k+1}u\|^2_{L^2}\right)\\
	& \leq  C\|u\|_{\dot{H}^{\beta(d,k)}\cap H^{\frac{(d-2)(k+1)}{2k}}}\left(\|\partial^k \bar{u}\|^2_{L^2}+\|\partial^{k+1}u\|^2_{L^2}\right).
	\end{split}
\end{equation*}
Summing those estimates concludes the desired result.
\end{proof}

\begin{theorem}\label{parabolic decay}
Let the assumptions of Theorem \ref{simple result} be satisfied. If we further assume $\|(u_0,\eta_0)\|_{H^{-s}}\leq \epsilon_1$ for some $\epsilon_1=\epsilon_1(s)>0$, where $0<-\beta(d,m)\leq s<\frac{d}{2}$ with $\beta(d,m)$ given as in \eqref{betadm}, then there exists a constant $C>0$ independent of $t$ such that 

	\begin{equation*}
	\|\Lambda^l (u,\eta)(t)\|_{L^2}\leq \frac{C}{(1+t)^\frac{s+l}{2}}
\end{equation*}
for any $-s< l \leq m$.
\end{theorem} 
	\begin{remark}\label{rmk_gd}
 Note that since $\beta(d,m)=0$ for $m=d=3$, we can take $s=0$, thus the smallness of the negative Sobolev norm is not necessary for Theorem \ref{parabolic decay}.
	\end{remark}
\begin{proof}[Proof of Theorem \ref{parabolic decay}]
For sufficiently small $\| (u,\eta)\|_{H^{d-2}\cap \dot{H}^{-s}} \ll1$, it follows from Lemmas \ref{lem_ns_pptt} and \ref{decay estimate for parabolic} that
 \begin{equation*}
\begin{split}
 \| (u,\eta)(t)\|^2_{H^{d-2}\cap\dot{H}^{-s}}+\int^t_0\|\nabla(u, \eta)(\tau)\|^2_{H^{d-2}\cap \dot{H}^{-s}}d\tau \leq \|(u_0,\eta_0)\|^2_{H^{d-2}\cap\dot{H}^{-s}}  
	\end{split}
\end{equation*}
and
\begin{equation}\label{simple ode}
	\begin{split}
		\frac{d}{dt}\|\partial^k (u,\eta)\|^2_{L^2}+\|\partial^k\nabla (u, \eta)\|^2_{L^2}\leq 0
	\end{split}
\end{equation}
for $k \in [2, m]$. To get the decay rate of convergence, in particular, we take $k=m$ and bound the second term on the left-hand side of the above from below as
\begin{equation*}
	\begin{split}
		\|\partial^m u\|_{L^2}\leq \|\partial^{m+1} u\|^\frac{m+s}{m+1+s}_{L^2}\|\Lambda^{-s}u\|^\frac{1}{m+1+s}_{L^2} \leq C_0\|\partial^{m+1} u\|^\frac{m+s}{m+1+s}_{L^2}
	\end{split}
\end{equation*}
for some $C_0 > 0$ independent of $t$.
Applying this inequality to (\ref{simple ode}) yields that
\begin{equation*}
	\begin{split}
		\frac{d}{dt}\|\partial^m (u,\eta)\|^2_{L^2}+C_1\|\partial^m (u, \eta)\|^{\frac{2(m+1+s)}{m+s}}_{L^2}\leq 0,
	\end{split}
\end{equation*}
and subsequently, by solving this differential inequality, we obtain the decay rates of the solutions
\begin{equation*}
	\|\partial^m (u,\eta)\|_{L^2}\leq \frac{C}{(t+1)^\frac{m+s}{2}},
\end{equation*}
where $-\frac{d}{2} <-s\leq \min\{0, \beta(d,m)\}$ and $C>0$ is independent of $t$.  Moreover, by using the interpolation inequality, Lemma \ref{big lemma} (ii), we have
\begin{equation*}
	\|\Lambda^l (u,\eta)(t)\|_{L^2}\leq \|\Lambda^m(u,\eta)(t)\|^\frac{l+s}{m+s} _{L^2}\|\Lambda^{-s}(u,\eta)(t)\|^\frac{m-l}{m+s} _{L^2}\leq \frac{C}{(t+1)^\frac{l+s}{2} }
\end{equation*} 
for any $-s< l \leq m$. This concludes the desired result.
\end{proof}

%
%
%
%
%
%
%
\section{Cauchy problem for the Toner--Tu model} \label{globals tt}
In this section, we consider the global Cauchy problem for the TT model  \eqref{original}.  Similar to the structure of the previous section, we first discuss the global-in-time existence of the unique regular solution to \eqref{original} and then present its large-time behavior in the following two subsections.

%
%
%
%
%
%
%
\subsection{Global-in-time existence of solutions}
In this subsection, we aim to prove the following existence theorem:
	\begin{theorem}\label{tt result}
	   Let $m\geq 3$ be an integer with $d=2,3$.  If $(u_0,\eta_0)\in H^m$ with $\|(u_0,\eta_0)\|_{H^3}\leq \epsilon_0$ for some $\epsilon_0>0$, then there exists the unique solution $(u,\eta)$ of  \eqref{original} satisfying 
\[
u\in C([0,\infty);H^m)\cap L^2(0,\infty; \dot{H}^{m+1})  \quad \mbox{and} \quad  \eta\in C([0,\infty);H^m) \cap L^2(0,\infty; \dot{H}^m).
\]
	Moreover, we have 
	\begin{equation*}
	\begin{split}
		\|( u, \eta)(t)\|^2_{H^m}+\int^t_0(\|\nabla (u,\eta)(s)\|^2_{H^{m-1}} +\|\nabla u(s)\|_{\dot H^m}^2 +\|\bar{u}(s)\|^2_{H^m})\,ds& \leq  C\|( u_0, \eta_0)\|^2_{H^m}
		\end{split}
\end{equation*}
for all $t\geq0$ and some $C>0$ independent of $t$.
\end{theorem}
\begin{remark}
Compared to the PPTT model, we require a stronger regularity for the smallness condition on the initial data. This is due to the lack of dissipation rate with respect to $\eta$. 
\end{remark}
\begin{remark}
Theorem \ref{tt result} gives the stability of the steady-state solutions in $H^m \times H^m$ for $m \geq 3$.
\end{remark}

%
%
%
%
%
%
%
Similarly as before, for the proof of Theorem \ref{tt result}, we only concentrate on the a priori estimates of solutions. 

\begin{lemma}\label{global tt} Let $T>0$, $m\geq 3$, and $(u,\eta) \in X_m(T) \times C([0,T]; H^m)$ be a solution to the system \eqref{original}.
Then we have
\begin{equation*}
	\begin{split}
		&\frac{1}{2}\frac{d}{dt}\|(u,\eta)\|^2_{H^m}+\frac{3}{4} \|\nabla u\|^2_{H^m}+2\|\bar{u}\|^2_{H^m}\leq C\left(\|(u,\eta)\|^4_{H^3}+\|(u,\eta)\|_{H^3}\right)\left(\|\nabla (u, \eta)\|^2_{H^{m-1}}+\|\bar{u}\|^2_{L^2}\right).	\end{split}
\end{equation*}	
	\end{lemma}
\begin{proof}
Taking the $L^2$ inner product of the system \eqref{original} with $(u,\eta)$ gives
\begin{equation*}
	\begin{split}
		\frac{1}{2}&\frac{d}{dt}\|(u,\eta)\|^2_{L^2}+\|\nabla u\|^2_{L^2}+\|\nabla\cdot u\|^2_{L^2}+2\|\bar{u}\|^2_{L^2}+\|u\|^4_{L^4}\\
		&=-\int (u\cdot\nabla u)\cdot u \,dx -3\int \bar{u}|u|^2\,dx-\int\eta\nabla\cdot(\eta u)\,dx\\
		&\quad +\int (u\cdot\nabla (u\cdot \nabla u)+u\cdot \nabla ({\rm e}_1\cdot\nabla u)+{\rm e}_1\cdot \nabla (u\cdot\nabla u)+{\rm e}_1\cdot \nabla ({\rm e}_1\cdot\nabla u))\cdot u \,dx\\
		&=: I+II+III+VI,
	\end{split}
\end{equation*}
where we use \eqref{L2 esimate for sm} to find
\begin{equation*}
\begin{split}
	I+II+III& \leq  C\|(u,\eta)\|_{H^{d-2}}(\|\bar{u}\|^2_{L^2}+\|\nabla u\|^2_{L^2}+\|\nabla \eta\|^2_{L^2}).
\end{split}	
\end{equation*}
For $IV$, we use the integration by parts to get
\begin{equation*}
	\begin{split}
		IV&=\int (u\cdot\nabla (u\cdot \nabla u)+u\cdot \nabla ({\rm e}_1\cdot\nabla u)+{\rm e}_1\cdot \nabla (u\cdot\nabla u)+{\rm e}_1\cdot \nabla ({\rm e}_1\cdot\nabla u))\cdot u \,dx\\
		& \leq C(\|u\|_{L^\infty}+\|u\|^2_{L^\infty})\|\nabla u\|^2_{L^2}-\int |{\rm e}_1\cdot\nabla u|^2 \,dx,
	\end{split}
\end{equation*}
and this yields
\begin{equation*}
	\frac{1}{2}\frac{d}{dt}\|(u,\eta)\|^2_{L^2}+\|\nabla u\|^2_{L^2}+\|\bar{u}\|^2_{L^2}+\|u\|^4_{L^4}\leq C\|(u,\eta)\|_{H^3}\left(\|\nabla (u,\eta)\|^2_{L^2}+\|\bar u\|^2_{L^2}\right).
\end{equation*}
For $k=1,\dots, m$, we get
\begin{equation}\label{repeat}
	\begin{split}
		\frac{1}{2}&\frac{d}{dt}\|\partial^k (u, \eta)\|^2_{L^2}+\|\partial^k\nabla u\|^2_{L^2}+\|\partial^k\nabla\cdot u\|^2_{L^2}+2\|\partial^k   \bar{u}\|^2_{L^2}\\
		&=-\int \partial^k(u \cdot\nabla u)\partial^k u\,dx-\int \partial^k(u |u|^2)\partial^k u \,dx-\int \partial^k( |u|^2)\partial^k\bar{u} \,dx\\
		&\quad -2\int \partial^k(\bar{u} u)\partial^k u \,dx- \int \partial^k\nabla\cdot(\eta  u)\partial^k \eta \,dx+\int \partial^k ((u+{\rm e}_1)\cdot\nabla ((u+{\rm e}_1)\cdot \nabla u))\partial^k u\,dx\\
		&=:I+II+III+IV+V+VI.
		\end{split}
\end{equation}
Here we use almost the same argument as in the proof of Lemma \ref{priori for parabolic} to obtain
\begin{equation*}
\begin{split}
		I\ls \|\nabla u\|_{L^\infty}\|\partial^k u\|^2_{L^2}, \quad II\ls \|u\|^2_{L^\infty}\|\partial^ku\|^2_{L^2},  \ \ \ \text{and} \  \  \  III+IV\ls \|u\|_{L^\infty}\|\partial^ku\|^2_{L^2}.
\end{split}
\end{equation*}
For the estimate of $V$, we apply Lemma \ref{big lemma} (iv) to show
\begin{equation}\label{term v}
\begin{split}
		V&=- \int \partial^k\nabla\cdot(\eta  u)\partial^k \eta \,dx\cr
		&=-\int \partial^k\eta (\partial^k(\nabla\eta\cdot u)+\partial^k(\eta\nabla\cdot u)) \,dx\\
	&= -\int \partial^k\eta [\partial^k,u\cdot\nabla]\eta \,dx-\int \partial^k\eta [\partial^k,\eta\nabla\cdot]u \,dx-\int \partial^k\eta (\partial^k\nabla\eta)\cdot u \,dx-\int (\partial^k\eta) \eta\partial^k\nabla\cdot u\,dx\\
	&\leq C\left(\|\nabla u \|_{L^\infty}\|\partial^k\eta\|_{L^2}+\|\nabla \eta\|_{L^\infty}\|\partial^k u\|_{L^2}\right)\|\partial^k \eta\|_{L^2}\\
	&\quad +C\|\nabla u \|_{L^\infty}\|\partial^k\eta\|^2_{L^2}+C\|\eta\|^2_{L^\infty}\|\partial^k \eta\|^2_{L^2}+\frac{1}{32}\|\partial^k\nabla u\|^2_{L^2}\\
	& \leq C(\|(u, \eta)\|_{H^3}+\|\eta\|^2_{H^3})\|\partial^k\eta\|^2_{L^2}+\frac{1}{32}\|\partial^k\nabla u\|^2_{L^2}.
	\end{split}
\end{equation}
We next express the term $VI$ as
\begin{equation*}
	\begin{split}
		VI&=\int \partial^k (u\cdot\nabla (u\cdot \nabla u))\partial^k u \,dx+\int \partial^k (u\cdot \nabla ({\rm e}_1\cdot\nabla u))\partial^k u \,dx\\
		&\quad +\int \partial^k({\rm e}_1\cdot \nabla (u\cdot\nabla u))\partial^k u \,dx+\int \partial^k ({\rm e}_1\cdot \nabla ({\rm e}_1\cdot\nabla u))\partial^k u \,dx\\
		&=:VI_1+VI_2+VI_3+VI_4,
	\end{split}
\end{equation*}
where we further split the term $VI_1$ into five terms as:
\begin{equation*}
	\begin{split}
		VI_1&= \int ([\partial^{k}, u\cdot\nabla] (u\cdot \nabla u))\cdot\partial^k u \,dx+\int (u \cdot\nabla \partial^k (u\cdot \nabla u))\cdot\partial^k u \,dx\\
		&=\int ([\partial^{k}, u\cdot\nabla] (u\cdot \nabla u))\cdot \partial^k u \,dx-\int (u \cdot\nabla\partial^k u)\cdot \partial^k (u\cdot \nabla u)\,dx-\int \nabla\cdot u \partial^k u\cdot\partial^k(u\cdot\nabla u)\,dx\\
		&=\int ([\partial^{k}, u\cdot\nabla] (u\cdot \nabla u))\cdot\partial^k u \,dx-\int (u \cdot\nabla\partial^k u)\cdot ([\partial^k, u\cdot \nabla] u) \,dx-\int (u \cdot\nabla\partial^k u)\cdot  (u\cdot \nabla\partial^k u) \,dx\\
		&\quad -\int \nabla\cdot u \partial^k u\cdot ([\partial^k,u\cdot\nabla] u) \,dx-\int \nabla\cdot u \partial^k u\cdot(u\cdot\nabla\partial^k u) \,dx\\
		&=:VI_1^1+VI_1^2+VI_1^3+VI_1^4+VI_1^5.
	\end{split}
\end{equation*}
Then we use Lemma \ref{big lemma} to estimate 
\begin{equation*}
	\begin{split}
		VI_1^1 & \ls\left(\|\nabla u\|_{L^\infty}\|\partial^k(u\cdot \nabla u )\|_{L^2}+\|\nabla(u\cdot \nabla u)\|_{L^{2d}}\|\partial^ku\|_{L^{\frac{2d}{d-1}}} \right)\|\partial^ku\|_{L^2}\\
		& \ls \|\nabla u\|_{L^\infty}\left(\|[\partial^k, u\cdot\nabla]u\|_{L^2}+\|u\|_{L^\infty}\|\nabla \partial^ku\|_{L^2}\right)\|\partial^k u\|_{L^2}\\
		&\quad + \left(\|[\partial,u\cdot \nabla] u\|_{L^{2d}}+\|u\|_{L^\infty}\|\partial^2u\|_{L^{2d}}\right)\|\partial^ku\|_{L^{\frac{2d}{d-1}}} \|\partial^ku\|_{L^2}\\
		& \ls \|\nabla u\|_{L^\infty}\left(\|\nabla u \|_{L^\infty}\|\partial^k u\|_{L^2}+\|u\|_{L^\infty}\|\nabla \partial^ku\|_{L^2}\right)\|\partial^k u\|_{L^2}\\
		& \quad +\left(\|[\partial,u\cdot \nabla] u\|_{L^{2d}}+\|u\|_{L^\infty}\|\partial^2u\|_{L^{2d}}\right)\|\partial^ku\|_{L^{\frac{2d}{d-1}}} \|\partial^ku\|_{L^2}\\
		& \leq  C\left(\|\nabla u\|^2_{L^\infty}+\| u \|^2_{L^\infty}\right)\|\partial^k u\|^2_{L^2}+\frac{1}{128}\|\nabla \partial^ku\|_{L^2}\\
		&\quad   + C\left(\|\nabla u\|_{L^\infty}\|\nabla u\|_{L^{2d}}+\|u\|_{L^\infty}\|\partial^2u\|_{L^{2d}}\right)^\frac{4}{3}\|\partial^k u\|^2_{L^2},   \cr 
		VI_1^2+VI_1^4+VI_1^5& \ls \|u\|_{L^\infty}\|\nabla\partial^k u\|_{L^2}\|\nabla u\|_{L^\infty}\|\partial^k u\|_{L^2}+\|\nabla u\|^2_{L^\infty}\|\partial^k u\|^2_{L^2}\\
		& \leq C \left(\|u\|^2_{L^\infty}\|\nabla u\|^2_{L^\infty}+\|\nabla u\|^2_{L^\infty}\right)\|\partial^k u\|^2_{L^2}+\frac{1}{128}\|\nabla\partial^k u\|^2_{L^2},
	\end{split}
\end{equation*}
and
\begin{equation*}
	\begin{split}
		VI_1^3 =-\|u\cdot\nabla\partial^k u\|^2_{L^2}.	
	\end{split}
\end{equation*}
Summing those estimates deduces
\begin{equation*}
\begin{split}
	VI_1&\leq C\left(\|u\|^2_{L^\infty}+\|\nabla u\|^2_{L^\infty}+\|u\|^2_{L^\infty}\|\nabla u\|^2_{L^\infty}+\|\nabla u\|^\frac{4}{3}_{L^\infty}\|\nabla u\|^\frac{4}{3}_{L^{2d}}+\|u\|^\frac{4}{3}_{L^\infty}\|\partial^2u\|^\frac{4}{3}_{L^{2d}}\right)\|\partial^k u\|^2_{L^2}\\
	 &\quad +\frac{1}{64}\|\nabla\partial^k u\|^2_{L^2}-\|u\cdot\nabla\partial^k u\|^2_{L^2}.	
\end{split}	
\end{equation*}
By using similar arguments as the above, we also obtain
\begin{equation*}
	\begin{split}
		VI_2&= \int ([\partial^k,u\cdot\nabla]({\rm e}_1\cdot\nabla u))\cdot \partial^ku \,dx+\int (u\cdot \nabla ({\rm e}_1\cdot\nabla \partial^k u))\cdot\partial^k u\,dx\\
		&=\int ([\partial^k,u\cdot\nabla]({\rm e}_1\cdot\nabla u))\cdot \partial^ku \,dx-\int (u\cdot\nabla\partial^ku)\cdot({\rm e}_1\cdot\nabla\partial^ku) \,dx-\int \nabla\cdot u ({\rm e}_1\cdot\nabla \partial^k u)\cdot\partial^k u\,dx\\
		& \leq C\left(\|\nabla u\|_{L^\infty}\|\nabla\partial^k u\|_{L^2}+\|\partial^2 u\|_{L^{2d}}\|\partial^k u\|_{L^\frac{2d}{d-1}}\right)\|\partial^ku\|_{L^2}-\int (u\cdot\nabla\partial^ku)\cdot({\rm e}_1\cdot\nabla\partial^ku) \,dx\\
		& \leq  C(\|\nabla u\|^2_{L^\infty}+\|\partial^2u\|^\frac{4}{3}_{L^{2d}})\|\partial^k u\|^2_{L^2}+\frac{1}{128}\|\nabla\partial^k u\|^2_{L^2} +\frac{1}{2}\|u\cdot\nabla\partial^k u\|^2_{L^2}+\frac{1}{2}\|{\rm e}_1\cdot\nabla\partial^ku\|^2_{L^2}, \cr
		VI_3 &=-\int ({\rm e}_1\cdot \nabla \partial^k u)\cdot([\partial^k,u\cdot\nabla] u) \,dx-\int ({\rm e}_1\cdot\nabla\partial^k u)\cdot(u\cdot\nabla\partial^k u) \,dx\\
		& \leq  C\|\nabla u\|^2_{L^\infty}\|\partial^k u\|^2_{L^2}+\frac{1}{128}\|\nabla\partial^{k}u\|^2_{L^2}+\frac{1}{2}\|u\cdot\nabla\partial^k u\|^2_{L^2}+\frac{1}{2}\|{\rm e}_1\cdot\nabla\partial^ku\|^2_{L^2},
	\end{split}
\end{equation*}
and 
\begin{equation*}
	\begin{split}
		VI_4 =-\|{\rm e}_1\cdot\nabla\partial^k u\|^2_{L^2}.
	\end{split}
\end{equation*}
Combining the above estimates, we get
\begin{equation}\label{last recall}
	\begin{split}
		VI &\leq   C(\|u\|^2_{L^\infty}+\|\nabla u\|^2_{L^\infty}+\|u\|^2_{L^\infty}\|\nabla u\|^2_{L^\infty}+\|\nabla u\|^\frac{4}{3}_{L^\infty}\|\partial^2u\|^\frac{4}{3}_{L^2} )\|\partial^k u\|^2_{L^2}\cr
		&\quad + C\left( \|u\|^\frac{4}{3}_{L^\infty}\|\partial^2u\|^\frac{4}{3}_{L^{2d}} +\|\partial^2u\|^\frac{4}{3}_{L^{2d}}\right)\|\partial^k u\|^2_{L^2} +\frac{1}{32}\|\nabla\partial^k u\|^2_{L^2} \cr
		&\leq C(\|u\|_{H^3} + \|u\|_{H^3}^4)\|\partial^k u\|^2_{L^2}  + \frac{1}{32}\|\nabla\partial^k u\|^2_{L^2}. 
	\end{split}
\end{equation}
Putting this all together concludes
\[
\frac{1}{2}\frac{d}{dt}\|(u,\eta)\|^2_{H^m}+\frac{3}{4} \|\nabla u\|^2_{H^m}+2\|\bar{u}\|^2_{H^m}\leq C\left(\|(u,\eta)\|^4_{H^3}+\|(u,\eta)\|_{H^3}\right)\left(\|\nabla (u, \eta)\|^2_{H^{m-1}}+\|\bar{u}\|^2_{L^2}\right),
\]
where $C>0$ is indenepdent of $t$.
\end{proof}

In order to close the energy estimates $\|(u,\eta)\|^2_{H^m}$, we need to have an appropriate dissipation rate for $\eta$. For this, we provide the following {\it hypocoercivity}-type estimate. 
\begin{lemma}\label{hyper} Let the assumptions of Lemma \ref{global tt} be satisfied. Then there exists $C>0$ independent of $t$ such that 
\begin{align*} 
	\frac{d}{dt}\int \partial^k u\cdot \partial^k \nabla \eta \,dx+\frac{3}{4}\|\partial^k\nabla\eta\|^2_{L^2}\leq C\left(\|(u,\eta)\|^4_{H^3}+1\right)\|\partial^k \nabla u\|^2_{H^1} + C\|\pa^k \bar u\|_{L^2}^2
\end{align*}
for $k=0,\dots,m-1$.
	\end{lemma}
\begin{proof}
We begin with the zeroth-order estimate:
\begin{equation*}
	\begin{split}
		\frac{d}{dt}\int u\cdot\nabla\eta \,dx&=\int u_t \cdot \nabla \eta  +  u \cdot \nabla \eta_t \,dx=:I+II,
	\end{split}
\end{equation*}
where we first estimate $I$ as
    \begin{align*}
	I&=-\int (u\cdot\nabla u+\nabla \eta -\Delta u-\nabla\nabla\cdot u +{\rm e_1}\cdot\nabla u +|u|^2u+2\bar{u}u+2\bar{u}{\rm e_1}+|u|^2{\rm e_1}\\
	&\hspace{4cm} + u\cdot\nabla (u\cdot \nabla u)+u\cdot \nabla ({\rm e_1}\cdot\nabla u)+{\rm e_1}\cdot \nabla (u\cdot\nabla u)+{\rm e_1}\cdot \nabla ({\rm e_1}\cdot\nabla u)) \cdot \nabla\eta \,dx\\
	& \leq -\int |\nabla\eta|^2\,dx + C\|\nabla \eta\|_{L^2}\|\nabla u\|_{L^2}\|u\|_{L^\infty} + C\|\Delta u\|^2_{L^2}+\frac{1}{32}\|\nabla\eta\|_{L^2}^2 + C\|\nabla u\|^2_{L^2}\\
	&\quad +\frac{1}{32}\|\nabla\eta\|_{L^2} + C\|u\|_{L^\infty}\|u\|^2_{L^4}\|\nabla \eta\|_{L^2} + C\|u\|^4_{L^4}+\frac{1}{32}\|\nabla\eta \|^2_{L^2} + C\|\bar{u}\|_{L^2}^2 + \frac{1}{32}\|\nabla\eta\|_{L^2}^2\\
	&\quad + C(\|u\|^2_{L^\infty}\|\nabla u\|^2_{L^\infty}+\|\nabla u\|^2_{L^\infty})\|\nabla u\|^2_{L^2} + C(\| u\|^2_{L^\infty}+1)^2\|\Delta u\|^2_{L^2}\\
	& \leq -\frac{7}{8}\int |\nabla\eta|^2\,dx + C(\|u\|^2_{H^3}+1)^2\|\nabla u\|^2_{H^1} + C(\|u\|^2_{H^3}+1)\|u\|^4_{L^4} + C\|\bar{u}\|_{L^2}^2.
	\end{align*}
	
We next estimate 
\[
	II=-\int \nabla\cdot u\eta_t \,dx=\int\nabla\cdot u\nabla\cdot(\eta u) \,dx+\int\nabla\cdot u ({{\rm e}_1}\cdot \nabla\eta) \,dx 
+\int\nabla\cdot u\nabla\cdot u  \,dx=:II_{1}+II_{2}+II_{3}.
\]
Here we obtain
\[
	II_1 \leq \|\nabla u\|_{L^2}\|\nabla \eta\|_{L^2}\| u \|_{L^\infty}+\|\nabla u\|^2_{L^2}\|\eta\|_{L^\infty} \leq C(\|u\|^2_{H^3}+\|\eta\|_{H^3})\|\nabla u\|^2_{L^2}+\frac{1}{32}\|\nabla\eta\|^2_{L^2}
\]
and
\[
	II_2+II_3 \leq C\|\nabla u\|^2_{H^1}+\frac{1}{32}\|\nabla\eta\|^2_{L^2}.
\]
Summing these estimates yields that
\begin{align*}
	\frac{d}{dt}\int  u\cdot  \nabla \eta\,dx+\frac{3}{4}\|\nabla\eta\|^2_{L^2}\leq C((\|u\|^2_{H^3}+1)^2+\|\eta\|_{H^3})\| \nabla u\|^2_{H^1}+C\|\bar{u}\|^2_{L^2}.
\end{align*}

For the higher-order derivative estimates, we find that for $k=1,\dots, m-1$ 
\begin{equation*}
	\begin{split}
		\frac{d}{dt}\int \partial^k u\cdot\partial^k  \nabla\eta \,dx&=\int \partial^k u_t \cdot \partial^k\nabla \eta +\partial^k u \cdot \partial^k\nabla \eta_t\,dx=:III+IV.
	\end{split}
\end{equation*}
By Lemma \ref{big lemma} (iii) and (iv), we estimate $III$ as 
\begin{align*}
	III&=-\int \partial^k(u\cdot\nabla u+\nabla \eta -\Delta u-\nabla\nabla\cdot u +{\rm e}_1\cdot\nabla u +|u|^2u+2\bar{u}u+2\bar{u}{\rm e}_1+|u|^2{\rm e}_1\\
	&\hspace{8cm} +(u+{\rm e}_1)\cdot\nabla ((u+{\rm e}_1)\cdot \nabla u)) \cdot \partial^k\nabla\eta \,dx\\
	& \leq  -\int |\partial^k\nabla\eta|^2\,dx+C\left(\|\nabla u\|_{L^\infty}+\| u\|_{L^\infty}\right)\|\partial^k u\|_{H^1}\|\partial^k\nabla\eta\|_{L^2}+C\|\partial^k\nabla u\|^2_{H^1}\|\partial^k\nabla\eta\|_{L^2}\\
	&\quad +C\left(\|u\|^2_{L^\infty}+\|u\|_{L^\infty}\right)\|\partial^k u\|_{L^2}\|\partial^k\nabla \eta\|_{L^2}+C\|\partial^k \bar{u}\|_{L^2}\|\partial^k\nabla \eta\|_{L^2}\\
	&\quad +\int ((u+{\rm e}_1)\cdot\nabla ((u+{\rm e}_1)\cdot \nabla u))\cdot \partial^k\nabla\eta \,dx\\
	& \leq  -\frac{7}{8}\int |\partial^k\nabla\eta|^2\,dx+C\|\partial^k \bar{u}\|^2_{L^2}+(\|\nabla u\|^2_{L^\infty}+\|u\|^4_{L^\infty}+\|u\|^2_{L^\infty})\| \partial^k u\|^2_{H^1}+C\|\partial^k\nabla  u\|^2_{H^1}\\
	&\quad +\int ((u+{\rm e}_1)\cdot\nabla ((u+{\rm e}_1)\cdot \nabla u))\cdot \partial^k\nabla\eta \,dx.
\end{align*}
For the last term on the right-hand side of the above inequality, we divide it into four terms:
\begin{equation*}
	\begin{split}
		&\int ((u+{\rm e}_1)\cdot\nabla ((u+{\rm e}_1)\cdot \nabla u))\cdot\nabla\partial^k\eta \,dx\\
		&=\int \partial^k (u\cdot\nabla (u\cdot \nabla u))\cdot\nabla\partial^k \eta \,dx+\int \partial^k (u\cdot \nabla ({\rm e}_1\cdot\nabla u))\cdot\partial^k\nabla\eta \,dx\\
		&\quad +\int \partial^k({\rm e}_1\cdot \nabla (u\cdot\nabla u))\cdot\nabla\partial^k \eta \,dx+\int \partial^k ({\rm e}_1\cdot \nabla ({\rm e}_1\cdot\nabla u))\cdot\partial^k\nabla \eta \,dx\\
		&=:III_1+III_2+III_3+III_4.
	\end{split}
\end{equation*}
By using Lemma \ref{big lemma} (iv), we get
\begin{equation*}
	\begin{split}
		III_1&
		=\int ([\partial^{k}, u\cdot\nabla] (u\cdot \nabla u))\cdot\partial^k\nabla\eta \,dx+\int (u \cdot\nabla \partial^k (u\cdot \nabla u))\cdot\partial^k\nabla \eta \,dx\\
		&=\int ([\partial^{k}, u\cdot\nabla ](u\cdot \nabla u))\cdot\partial^k\nabla \eta \,dx
		+\int u_i (  [\partial_i \partial^k, u\cdot \nabla] u)\cdot\partial^k\nabla\eta \,dx+\int u_i(u\cdot\partial_i \nabla\partial^{k} u)\cdot   \partial^k\nabla\eta \,dx\\
		&\leq C\left(\|\nabla u\|_{L^\infty}\|\partial^k(u\cdot \nabla u )\|_{L^2}+\|\nabla(u\cdot \nabla u)\|_{L^{2d}}\|\partial^ku\|_{L^{\frac{2d}{d-1}}} \right)\|\partial^k\nabla\eta\|_{L^2}\\
		 &\quad +C\|u\|_{L^\infty}\|\nabla u\|_{L^\infty}\|\partial^k\nabla u\|_{L^2}\|\partial^k\nabla \eta\|_{L^2}+\|u\|^2_{L^\infty}\|\partial^{k+2} u\|_{L^2}\|\partial^k\nabla \eta\|_{L^2}\\
		 &\leq C\|\nabla u\|^2_{L^\infty}\left(\|[\partial^k,u\cdot \nabla] u )\|_{L^2}+\|u\|_{L^\infty}\|\partial^k\nabla u\|_{L^2}\right)^2\\
		 &\quad +C\left(\|[\partial,u\cdot \nabla] u\|_{L^{2d}}+\|u\|_{L^\infty}\|\partial^2u\|_{L^{2d}}\right)^2\left(\|\partial^ku\|^2_{L^2}+\|\partial^k\nabla u\|^2_{L^2}\right)\\
		 &\quad +C\|u\|^4_{H^3}\|\partial^k\nabla u\|^2_{H^1}
		 +\frac{1}{64}\|\partial^k\nabla \eta\|^2_{L^2}\\
		 &\leq C(\|\nabla u\|^2_{L^\infty}+\| u\|^2_{L^\infty})\|u\|^2_{H^3}\|\partial^k u\|^2_{H^1}+C\|u\|^4_{H^3}\|\partial^k\nabla u\|^2_{H^1}
		 +\frac{1}{64}\|\partial^k\nabla \eta\|^2_{L^2},\cr
		III_2&= 
		\int ([\partial^k,u\cdot\nabla]({\rm e}_1\cdot\nabla u))\cdot \partial^k\nabla\eta\, dx+\int (u\cdot\nabla({\rm e}_1\cdot\nabla\partial^k u)) \partial^k\nabla\eta \,dx\\
		&\leq C\left(\|\nabla u\|_{L^\infty}\|\partial^k\nabla u\|_{L^2}+\chi_{k\geq2}\|\partial^2 u\|_{L^{2d}}\|\partial^k u\|_{L^{\frac{2d}{d-1}}}\right)\|\partial^k\nabla\eta\|_{L^2}  +C\|u\|^2_{H^3}\|\partial^k\nabla u\|^2_{H^1}+\frac{1}{64}\|\partial^k\nabla\eta\|^2_{L^2}\\
		&\leq C( \|\nabla u\|^2_{L^\infty}+\chi_{k\geq2}\|\partial^2u\|^2_{L^{2d}})\|\partial^k u\|^2_{H^1}+C\|u\|^2_{H^3}\|\partial^k\nabla u\|^2_{H^1}+\frac{1}{64}\|\partial^k\nabla\eta\|^2_{L^2},	\cr
		III_3
		&=\int ([\partial_1\partial^k, u\cdot\nabla] u) \cdot\partial^k \nabla  \eta \,dx+\int (u\cdot\nabla\partial_1\partial^k u)\cdot\partial^k\nabla\eta \,dx\\
		&\leq  C\|\nabla u\|^2_{L^\infty}\|\partial^k\nabla u\|^2_{L^2}+C\|u\cdot\partial_1\partial^k\nabla u\|^2_{L^2}+\frac{1}{64}\|\partial^k\nabla\eta\|^2_{L^2},
    \end{split}
\end{equation*}
and 
\begin{equation*}
	\begin{split}
		III_4=\int \partial^k ({\rm e}_1\cdot \nabla ({\rm e}_1\cdot\nabla u))\cdot\nabla\partial^k \eta  \,dx\leq C\| \partial^{k+2} u\|^2_{L^2}+\frac{1}{64}\|\partial^k\nabla \eta\|^2_{L^2}.
	\end{split}
\end{equation*}
Summing those estimates and using Lemma \ref{big lemma} (vi), Lemma \ref{lem_gd1} (ii), (iii), and (v), we have
\begin{equation}\label{hypo damping}
	\begin{split}
		III&\leq  -\frac{7}{8}\|\partial^k\nabla\eta\|^2_{L^2}+ C( \|\nabla u\|^2_{L^\infty}+\| u\|^2_{L^\infty}+\chi_{k\geq2}\|\partial^2u\|^2_{L^{2d}})(\|u\|^2_{H^3}+1)\|\partial^k u\|^2_{L^2}\\
		&\quad +C(\|u\|^4_{H^3}+1)\|\partial^k\nabla u\|^2_{H^1}+C\|\partial^k\bar{u}\|^2_{L^2}+\frac{1}{16}\|\partial^k\nabla\eta\|^2_{L^2}\\
		&\leq  -\frac{7}{8}\|\partial^k\nabla\eta\|^2_{L^2}+ C( \| u\|^2_{H^\frac{(d-2)(k+2)}{2k}}+\|u\|^2_{H^3})(\|u\|^2_{H^3}+1)\|\partial^k\nabla u\|^2_{H^1}\\
		&\quad +C(\|u\|^4_{H^3}+1)\|\partial^k\nabla u\|^2_{H^1}+C\|\partial^k\bar{u}\|^2_{L^2}+\frac{1}{16}\|\partial^k\nabla\eta\|^2_{L^2}\\
		&\leq  -\frac{7}{8}\|\partial^k\nabla\eta\|^2_{L^2}+C(\|u\|^4_{H^3}+1)\|\partial^k\nabla u\|^2_{H^1}+C\|\partial^k\bar{u}\|^2_{L^2}+\frac{1}{16}\|\partial^k\nabla\eta\|^2_{L^2}.
	\end{split}
\end{equation}
Finally, we estimate $IV$ as 
\begin{align*}
    IV&=\int\partial^k\nabla\cdot u\partial^k\nabla\cdot(\eta u)\,dx +\int\partial^k\nabla\cdot u {\rm e}_1\cdot \nabla\partial^k\eta \,dx 
+\int\nabla\cdot\partial^k u\nabla\cdot\partial^k u\,dx =: IV_1+ IV_2+ IV_3,
\end{align*}
where
\begin{align*}
	 IV_1 &\ls \|\partial^k \nabla u\|_{L^2}\left(\|\eta\|_{L^\infty}\|\partial^k\nabla u\|_{L^2}+\|u\|_{L^\infty}\|\partial^k\nabla \eta\|_{L^2}\right) \leq  C\left(\|(u,\eta)\|^2_{H^3}+1\right)\|\partial^k\nabla u\|^2_{L^2}+\frac{1}{32}\|\partial^k\nabla \eta\|^2_{L^2}
\end{align*}
and
\begin{align*}
	 IV_2+ IV_3\leq C\|\partial^k\nabla u\|^2_{L^2}+\frac{1}{32}\|\nabla\partial^k\eta\|^2_{L^2}.
\end{align*}
Thus, we find
\begin{equation}\label{hypo damping2}
	\begin{split}
		 IV\leq C\left(\|(u,\eta)\|^2_{H^3}+1\right)\|\partial^k\nabla u\|^2_{L^2}+\frac{1}{16}\|\partial^k\nabla \eta\|^2_{L^2}.
	\end{split}
\end{equation}
We then combine \eqref{hypo damping} and \eqref{hypo damping2} to conclude
\begin{align*} 
	\frac{d}{dt}\int \partial^k u\cdot \partial^k \nabla \eta \,dx+\frac{3}{4}\|\partial^k\nabla\eta\|^2_{L^2}\leq C\left(\|(u,\eta)\|^4_{H^3}+1\right)\|\partial^k \nabla u\|^2_{H^1}+ C\|\partial^k\bar{ u}\|^2_{L^2}
\end{align*}
for some $C>0$ independent of $t$. This completes the proof.
\end{proof}

\begin{proof}[Proof of Theorem \ref{tt result}]
We first notice that there exists a positive constant $C_0 > 0$ such that
\[
C_0\| (u, \eta)\|_{H^m}^2 \leq \| (u, \eta)\|_{H^m}^2+\delta_0\sum^{m-1}_{k=0}\int \partial^ku\cdot\nabla\partial^k\eta \,dx \leq \frac1{C_0}\| (u, \eta)\|_{H^m}^2
\]
for $\delta_0 > 0$ small enough. On the other hand, it follows from Lemmas \ref{global tt} and \ref{hyper} that
\begin{align*}
	&\frac{d}{dt}\left(\| (u, \eta)\|_{H^m}^2 + \delta_0\sum^{m-1}_{k=0}\int \partial^ku\cdot\partial^k\nabla\eta \,dx\right)+\frac{3}{4}\delta_0\|\nabla\eta\|^2_{H^{m-1}}+\|\nabla u\|^2_{H^m}+\|\bar{u}\|^2_{H^m}\\
	&\quad \leq  C\delta_0(\|(u,\eta)\|_{H^3}+1)^2\|\nabla u\|_{H^m}+ C\left(\|(u,\eta)\|^4_{H^3}+\|(u,\eta)\|_{H^3}\right)\|\nabla (u, \eta)\|^2_{H^{m}}+C\delta_0\|\bar{u}\|^2_{H^m}.
\end{align*}
If we assume that
\begin{equation*}
	\sup_{0 \leq t \leq T}\|(u,\eta)(t)\|_{H^3}\ll 1,
\end{equation*}
then we get
\begin{align*}
	&\left(\|( u, \eta)\|_{H^m}^2+\delta_0\sum^{m-1}_{k=0}\int \partial^ku\cdot\partial^k\nabla\eta \,dx\right) + \frac12\int^t_0 \lt(\delta_0\|\nabla\eta (s)\|^2_{H^{m-1}} + \|\nabla u(s)\|^2_{H^m} + \|\bar{u}(s)\|^2_{H^m}\rt)ds\\
	&\quad \leq  C\| (u_0,\eta_0)\|_{H^m}.
\end{align*}
Hence, we conclude that 
\begin{equation*}
	\begin{split}
		&\|(u,\eta)(t)\|^2_{H^m}+\int^t_0 \left(\|\nabla (u,\eta)(s)\|^2_{H^{m-1}} +\|\nabla u(s)\|_{\dot H^m}^2 +\|\bar{u}(s)\|^2_{H^m}\right)ds\leq C\|(u_0,\eta_0)\|^2_{H^m}
	\end{split}
\end{equation*}
for all $T\geq t\geq0$ and some $C>0$ independent of $t$. From the above estimates together with the local-in-time existence theory, we conclude the desired global-in-time existence of solutions.
\end{proof}

%
%
%
%
%
%
%
\subsection{Large-time behavior}
In this part, we discuss the large-time behavior of solutions constructed in Theorem \ref{tt result}.

In parallel with Section \ref{sec_lt_pptt}, we begin with the estimate of the negative Sobolev norm of solutions.
\begin{lemma}\label{hypo decay} Let $0<s<\frac{d}{2}$,  $T>0$, and the assumptions of Lemma \ref{global tt} be satisfied. If we further assume that $u,\eta \in C([0,T]; \dot{H}^{-s})$ and $u \in L^2(0,T; \dot H^{1-s})$, then there exists $C>0$ independent of $T$ such that
\begin{equation*}
	\begin{split}
		&\frac{1}{2}\frac{d}{dt}\|  (u,\eta)\|^2_{\dot{H}^{-s}} + \| \nabla u\|^2_{\dot{H}^{-s}}+\|\bar{u}\|^2_{\dot{H}^{-s}}  \leq  C  (\|(u,\eta)\|_{H^2\cap\dot{H}^{-s}}+\|(u, \eta)\|^2_{H^2\cap\dot{H}^{-s}})\|\nabla (u,\eta)\|^2_{H^2\cap \dot{H}^{-s}}+\frac{1}{2}\|\bar{u}\|^2_{L^2} .
	\end{split}
\end{equation*}
\end{lemma}
\begin{proof}
Similarly as before, we first find
\begin{equation*} 
    \begin{split}
		&	\frac{1}{2}\frac{d}{dt}\|  (u,\eta)\|^2_{\dot{H}^{-s}} + \| \nabla u \|^2_{\dot{H}^{-s}}+\| \nabla\cdot u\|^2_{\dot{H}^{-s}} + 2\|\bar u\|_{\dot H^{-s}}^2\\
	& \quad =  -\int \Lambda^{-s}\nabla\cdot(\eta u)\Lambda^{-s}\eta \,dx-\int \Lambda^{-s}(u\cdot\nabla u)\cdot\Lambda^{-s}u \,dx-\int \Lambda^{-s}(|u|^2 u)\cdot\Lambda^{-s}u \,dx\\
	&\qquad -\int \Lambda^{-s}| u|^2\Lambda^{-s}\bar{u}\,dx-2\int \Lambda^{-s}(\bar u u)\cdot\Lambda^{-s}u \,dx\\
	&\qquad +\int \Lambda^{-s}(u\cdot\nabla (u\cdot \nabla u)+u\cdot \nabla ({\rm e}_1\cdot\nabla u)+{\rm e}_1\cdot \nabla (u\cdot\nabla u)+{\rm e}_1\cdot \nabla ({\rm e}_1\cdot\nabla u))\cdot\Lambda^{-s} u\, dx\\
	&\quad =:  I+II+III+IV+V+VI,
    \end{split}
\end{equation*}
where we use the estimates in the proof of Lemma \ref{lem_ns_pptt} to get
\begin{equation*}
\begin{split}
	&I+II+III+IV+V\\
	&\quad \leq  C\|\nabla (u,\eta)\|_{H^{d-2}\cap\dot{H}^{-s}}^2\|(u,\eta)\|_{\dot H^{-s}}(1 + \|(u,\eta)\|_{\dot H^{-s}}) + \frac12 \|\bar u\|_{L^2}^2\\
	&\qquad +C\|\nabla u\|_{H^{d-2}\cap\dot{H}^{-s}}^2\|u\|_{L^2}^2 + \frac12 \|\bar u\|_{\dot H^{-s}}^2 +C \| u\|^2_{H^{d-2}\cap\dot{H}^{-s}} \|\nabla u\|^2_{H^{d-2}\cap\dot{H}^{-s}}\\
	&\quad \leq C (\|(u,\eta)\|_{H^{d-2}\cap\dot{H}^{-s}}+  \|(u,\eta)\|^2_{H^{d-2}\cap\dot{H}^{-s}}  )\|\nabla (u,\eta)\|^2_{H^{d-2}\cap \dot{H}^{-s}} +\frac{1}{2}\|\bar{u}\|^2_{L^2\cap\dot{H}^{-s}}.
\end{split}
\end{equation*}
For $VI$, we divide it into four terms:
\begin{equation*}
	\begin{split}
	VI&=\int \Lambda^{-s}(u\cdot\nabla (u\cdot \nabla u)+u\cdot \nabla ({\rm e}_1\cdot\nabla u)+{\rm e}_1\cdot \nabla (u\cdot\nabla u)+{\rm e}_1\cdot \nabla ({\rm e}_1\cdot\nabla u))\cdot\Lambda^{-s} u \,dx\\
	&=:VI_1+VI_2+VI_3+VI_4
	\end{split}
\end{equation*}
and estimate the each term as
\begin{equation*}
	\begin{split}
		VI_1
		& \ls  \|\Lambda^{-s}(u\cdot\nabla (u\cdot \nabla u))\|_{L^2}\|\Lambda^{-s}u\|_{L^2}\cr
		& \ls \|u\cdot\nabla (u\cdot \nabla u)\|_{L^{\frac{1}{\frac{1}{2}+\frac{s}{d}}}}\|\Lambda^{-s}u\|_{L^2}\\
		&\ls \|\Lambda^{\frac{d}{2}-s}u\|_{L^2}(\|[\partial, u\cdot\nabla] u\|_{L^2}+\| u\cdot\partial \nabla u\|_{L^2})\|\Lambda^{-s}u\|_{L^2}\\
		& \ls \|\Lambda^{\frac{d}{2} -s}u\|_{L^2}(\| u\|_{L^\infty}\| \partial\nabla u\|_{L^2}+\|\nabla u\|_{L^4}\|\nabla u\|_{L^4})\|\Lambda^{-s}u\|_{L^2}\\
		& \ls \|u\|_{H^2}\|\nabla u\|^2_{H^2\cap \dot{H}^{-s}}\|u\|_{\dot{H}^{-s}}, \cr
		VI_2+VI_3
		& \ls \left(\|\partial u\cdot\nabla u\|_{L^{\frac{1}{\frac{1}{2}+\frac{s}{d}}}}+\|u \cdot\partial\nabla u\|_{L^{\frac{1}{\frac{1}{2}+\frac{s}{d}}}}\right)\|u\|_{\dot{H}^{-s}}\\
		& \ls \|\nabla u\|^2_{H^2\cap \dot{H}^{-s}}\|u\|_{\dot{H}^{-s}},
	\end{split}
\end{equation*}
and
\begin{equation*}
	\begin{split}
		IV_4 =\int({\rm e}_1\cdot \nabla ({\rm e}_1\cdot\nabla \Lambda^{-s} u))\Lambda^{-s}u \,dx =-\int|{\rm e}_1\cdot\nabla \Lambda^{-s} u|^2 \,dx
	\end{split}
\end{equation*}
due to Lemma \ref{big lemma} (v) and (vi).

We finally combine all of the above estimates to have
\begin{equation*}
	\begin{split}
		&\frac{1}{2}\frac{d}{dt}\|  (u,\eta)\|^2_{\dot{H}^{-s}} + \| \nabla u\|^2_{\dot{H}^{-s}}+\|\bar{u}\|^2_{\dot{H}^{-s}}\ls    (\|(u,\eta)\|_{H^2\cap\dot{H}^{-s}}+\|(u, \eta)\|^2_{H^2\cap\dot{H}^{-s}})\|\nabla (u,\eta)\|^2_{H^2\cap \dot{H}^{-s}}+\frac{1}{2}\|\bar{u}\|^2_{L^2} .
\end{split}
\end{equation*}
This completes the proof.
\end{proof}

In the lemma below, we also provide the estimate giving the dissipation rate for $\eta$. 

\begin{lemma}\label{hypo decay1}Let the assumptions of Lemma \ref{hypo decay} be satisfied. Then there exists $C>0$ independent of $T$ such that
\begin{equation*}
	\begin{split}
	\frac{d}{dt}\int \Lambda^{-s} u\cdot \Lambda^{-s} \nabla \eta \,dx+\frac{3}{4}\|\nabla\eta\|^2_{\dot H^{-s}}\leq C(\|(u,\eta)\|^4_{H^3}+1)\| \Lambda^{1-s} u\|^2_{H^3}+C\|\bar{u}\|^2_{\dot H^{-s}}.	
	\end{split}
\end{equation*}
\end{lemma}
\begin{proof}
It follows from \eqref{original} that 
\begin{equation*}
	\begin{split}
		\frac{d}{dt}\int \Lambda^{-s} u\cdot\Lambda^{-s}\nabla\eta \,dx&=\int \Lambda^{-s}u_t \cdot \Lambda^{-s}\nabla \eta +\Lambda^{-s}u \cdot \Lambda^{-s}\nabla \eta_t \,dx=:I+II,
	\end{split}
\end{equation*}
where
\begin{align*}
	I&=-\int \Lambda^{-s}(u\cdot\nabla u+\nabla \eta -\Delta u-\nabla\nabla\cdot u +{\rm e}_1\cdot\nabla u +|u|^2u+2\bar{u}u+2\bar{u}{\rm e}_1+|u|^2{\rm e}_1\\
	&\hspace{3cm} +u\cdot\nabla (u\cdot \nabla u)+u\cdot \nabla ({\rm e}_1\cdot\nabla u)+{\rm e}_1\cdot \nabla (u\cdot\nabla u)+{\rm e}_1\cdot \nabla ({\rm e}_1\cdot\nabla u)) \cdot \Lambda^{-s}\nabla\eta \,dx\\
	& \leq -\int |\Lambda^{-s}\nabla\eta|^2\,dx+\frac{1}{32}\|\Lambda^{-s}\nabla \eta\|^2_{L^2}+C\|\nabla u\|^2_{L^2}\|\nabla\Lambda^{-s} u\|^2_{H^1}+C\|\nabla\Lambda^{-s} u\|^2_{H^1}\\
	&\quad  +C\|\Lambda^{\frac{d}{2} -s}u\|^2_{L^2}\| u\|^2_{H^{\frac{3}{2}-\frac{3}{d}}}\|\nabla u\|^2_{L^2}+ C\|\bar{u}\|^2_{L^2}\|\Lambda^{\frac{d}{2}-s} u\|^2_{L^2}+C\|\Lambda^{-s}\bar{u}\|^2_{L^2}+C\|u\|^2_{L^2}\|\Lambda^{\frac{d}{2}-s} u\|^2_{L^2}\\
	&\quad  -\int \Lambda^{-s}(u\cdot\nabla (u\cdot \nabla u)+u\cdot \nabla ({\rm e}_1\cdot\nabla u)+{\rm e}_1\cdot \nabla (u\cdot\nabla u)+{\rm e}_1\cdot \nabla ({\rm e}_1\cdot\nabla u))\cdot \Lambda^{-s}\nabla\eta \,dx.
\end{align*}
Here the last term on the right-hand side of the above inequality can be estimated as
\begin{equation*}
\begin{split}
	&\int \Lambda^{-s}(u\cdot\nabla (u\cdot \nabla u)+u\cdot \nabla ({\rm e}_1\cdot\nabla u)+{\rm e}_1\cdot \nabla (u\cdot\nabla u)+{\rm e}_1\cdot \nabla ({\rm e}_1\cdot\nabla u))\cdot \Lambda^{-s}\nabla\eta \,dx  =:I_1+I_2+I_3+I_4
	\end{split}
\end{equation*}
with 
\begin{equation*}
	\begin{split}
		I_1 & \leq  \|\Lambda^{-s}(u\cdot\nabla (u\cdot \nabla u))\|_{L^2}\|\Lambda^{-s}\nabla\eta\|_{L^2}\cr
		&\ls \|u\cdot\nabla (u\cdot \nabla u)\|_{L^{\frac{1}{\frac{1}{2}+\frac{s}{d}}}}\|\Lambda^{-s}\nabla\eta\|_{L^2}\\
		&\ls \|\Lambda^{\frac{d}{2} -s}u\|_{L^2}(\|[\partial, u\cdot\nabla] u\|_{L^2}+\| u\cdot\partial\nabla u\|_{L^2})\|\Lambda^{-s}\nabla\eta\|_{L^2}\\
		& \ls \|\Lambda^{\frac{d}{2} -s}u\|_{L^2}(\| u\|_{L^\infty}\| \partial^2 u\|_{L^2}+\|\nabla u\|^2_{L^4})\|\Lambda^{-s}\nabla\eta\|_{L^2}\cr
		&\ls \|u\|^2_{H^3}\|\nabla u\|_{H^1}\|\Lambda^{-s}\nabla\eta\|_{L^2} \leq C\|u\|^4_{H^3}\|\nabla u\|^2_{H^1}+\frac{1}{32}\|\nabla\eta\|_{\dot{H}^{-s}}, \cr
		I_2+I_3 & \ls \left(\|(u\cdot \partial\nabla u)\|_{L^{\frac{1}{\frac{1}{2}+\frac{s}{d}}}}+ \||\nabla u|^2\|_{L^{\frac{1}{\frac{1}{2}+\frac{s}{d}}}}\right)\|\Lambda^{-s}\nabla\eta\|_{L^2} \cr
		  &\ls \|u\|_{H^3}\|\nabla u\|_{H^1}\|\Lambda^{-s}\nabla\eta\|_{L^2} \leq C\|u\|^2_{H^3}\|\nabla u\|^2_{H^1}+\frac{1}{32}\|\nabla\eta\|_{\dot{H}^{-s}},
	\end{split}
\end{equation*}
and
\begin{equation*}
	\begin{split}
		I_4\ls \|\Lambda^{-s}\partial^2 u\|_{L^2}\|\Lambda^{-s}\nabla\eta\|_{L^2} \leq  C\|\Lambda^{-s}\partial^2 u\|^2_{L^2}+\frac{1}{32}\|\nabla\eta\|^2_{\dot{H}^{-s}}
	\end{split}
\end{equation*}
thanks to Lemma \ref{big lemma} (i), (v), and (vi).

Analogously, we also estimate $II$ as 
\begin{align*}
	II &=\int\Lambda^{-s}\nabla\cdot u\Lambda^{-s}\nabla\cdot(\eta u) \,dx +\int\Lambda^{-s}\nabla\cdot u {\rm e}_1\cdot \Lambda^{-s}\nabla\eta \, dx+\int\Lambda^{-s}\nabla\cdot u\Lambda^{-s}\nabla\cdot u \,dx\\
	& \leq C\left(\|\nabla\eta\cdot u\|_{L^{\frac{1}{\frac{1}{2}+\frac{s}{d}}}}+\|\eta\nabla\cdot u\|_{L^{\frac{1}{\frac{1}{2}+\frac{s}{d}}}}\right)\|\Lambda^{1-s}u\|_{L^2}+ C\|\Lambda^{1-s}u\|^2_{L^2}+\frac{1}{32}\|\Lambda^{1-s}\eta\|^2_{L^2} \\
	& \leq C\left(\|\nabla\eta\|_{L^2}\|\Lambda^{\frac{d}{2}-s} u\|_{L^2}+\|\Lambda^{\frac{d}{2}-s}\eta\|_{L^2}\|\nabla u\|_{L^{2}}\right)\|\Lambda^{1-s}u\|_{L^2}+ C\|\Lambda^{1-s}u\|^2_{L^2}+\frac{1}{32}\|\Lambda^{1-s}\eta\|^2_{L^2} \\
	& \leq  C\|\eta\|_{H^3}\|\Lambda^{1-s}u\|^2_{H^3}+C \|\Lambda^{1-s}u\|^2_{L^2}+\frac{1}{32}\|\Lambda^{1-s}\eta\|^2_{L^2}.
\end{align*}
Combining those estimates yields
\begin{align*}
	\frac{d}{dt}\int \Lambda^{-s} u\cdot \Lambda^{-s} \nabla \eta\, dx+\frac{3}{4}\|\Lambda^{-s}\nabla\eta\|^2_{L^2}\leq C(\|(u,\eta)\|^4_{H^3}+1)\| \Lambda^{1-s} u\|^2_{H^3}+C\|\Lambda^{-s}\bar{u}\|^2_{L^2}
\end{align*}
for some $C>0$ independent of $T$.
\end{proof}

%
%
%
%
%
%
%

Parallel to Lemma \ref{decay estimate for parabolic}, we modify the energy estimate in Lemma \ref{global tt} by making use of the negative Sobolev space $\dot H^{\beta(d,m)}$.

\begin{lemma}\label{ttdecay lem1} 
Let $T>0$, $m\geq 3$, $\max\{0, -\beta(d,m)\} \leq s<\frac{d}{2}$ with $\beta(d,m)$ given as in \eqref{betadm}, and the assumptions of Lemma \ref{hypo decay} be satisfied.  Then we have
\begin{equation*}
	\begin{split}
		\frac{1}{2}&\frac{d}{dt}\|\partial^{m-1} (u, \eta)\|^2_{H^1}+\|\partial^{m-1}\nabla u\|^2_{H^1}+2\|\partial^{m-1}   \bar{u}\|^2_{H^1}\\
		& \leq \left(C\left(\|u\|_{H^3\cap\dot{H}^{\beta(d,m)}}+\|u\|^2_{H^3\cap\dot{H}^{\beta(d,m)}}
		+\|\eta\|_{H^3}\right)+\frac{3}{32}\right)\left(\| \partial^{m-1}\nabla u\|^2_{H^1}+\|\partial^{m-1}\bar{u}\|^2_{H^1}\right)\\
		&\quad +C(\|(u,\eta)\|_{H^3}+\|(u,\eta)\|^2_{H^3})\|\partial^{m-1}\nabla \eta\|^2_{L^2}.
		\end{split}
\end{equation*}
\end{lemma}
\begin{proof}
We first recall the estimate \eqref{repeat} from the proof of Lemma \ref{global tt}:
\begin{equation*}
	\begin{split}
		&\frac{1}{2}\frac{d}{dt}\|\partial^k (u, \eta)\|^2_{L^2}+\|\partial^k\nabla u\|^2_{L^2}+\|\partial^k\nabla\cdot u\|^2_{L^2}+2\|\partial^k   \bar{u}\|^2_{L^2}\\
		&\quad =-\int \partial^k(u \cdot\nabla u)\partial^k u\,dx-\int \partial^k(u |u|^2)\partial^k u \,dx-\int \partial^k( |u|^2)\partial^k\bar{u} \,dx\\
		&\qquad -2\int \partial^k(\bar{u} u)\partial^k u \,dx- \int \partial^k\nabla\cdot(\eta  u)\partial^k \eta \,dx+\int \partial^k ((u+{\rm e}_1)\cdot\nabla ((u+{\rm e}_1)\cdot \nabla u))\partial^k u\,dx\\
		&\quad =:I+II+III+IV+V+VI
		\end{split}
\end{equation*}
for $k=m-1,m$, where
\begin{equation*}
\begin{split}
&I+II+III+IV\cr
&\quad \leq C(\| u\|_{{H^1}\cap \dot{H}^{\beta(d,k)}}+\|u\|^2_{H^1})(\|\partial^{k+1} u\|^2_{L^2}+\|\partial^k \bar{u}\|^2_{L^2}) +\frac{1}{32}\left(\|\partial^k\nabla u\|^2_{L^2}+\|\partial^k \bar{u}\|^2_{L^2}\right).
\end{split}
\end{equation*}
For $V$, we use \eqref{term v} together with applying Lemma \ref{lem_gd1} (i) and (ii) to show that
\begin{equation*}
\begin{split}
		V& \leq C\left(\|\nabla u \|_{L^\infty}\|\nabla\partial^k\eta\|_{L^2}+\|\nabla \eta\|_{L^\infty}\|\partial^k u\|_{L^2}\right)\|\partial^k \eta\|_{L^2}\\
	&\quad +C\|\nabla u \|_{L^\infty}\|\partial^k\eta\|^2_{L^2}+C\|\eta\|^2_{L^\infty}\|\partial^k \eta\|^2_{L^2}+\frac{1}{32}\|\partial^k\nabla u\|^2_{L^2}\\
	& \leq 	C\left(\|(u ,\eta)\|_{H^3}+\|(u ,\eta)\|^2_{H^3}\right)\|\partial^{m-1}\nabla\eta\|^2_{L^2}+C\|(u,\eta)\|_{H^3}\|\partial^k\nabla u\|^2_{L^2}+\frac{1}{32}\|\partial^k\nabla u\|^2_{L^2}. 
	 \end{split}
\end{equation*}
We next use \eqref{last recall} to get
\begin{equation*}
	\begin{split}
		VI 		& \leq C\left(\|u\|_{H^3}+\|u\|^4_{H^3}\right)\|\nabla\partial^k u\|^2_{L^2}	+\frac{1}{32}\|\nabla\partial^k u\|^2_{L^2}.
	\end{split}
\end{equation*}
Finally, we combine all of the above estimates to have
\begin{equation*}
	\begin{split}
		&\frac{1}{2}\frac{d}{dt}\|\partial^{m-1} (u, \eta)\|^2_{H^1}+\|\partial^{m-1}\nabla u\|^2_{H^1}+2\|\partial^{m-1}   \bar{u}\|^2_{H^1}\\
		&\quad \leq \left(C\left(\|u\|_{H^3\cap \dot{H}^{\beta(d,m)}}+\|u\|^2_{H^3\cap \dot{H}^{\beta(d,m)}}
		+\|\eta\|_{H^3}\right)+\frac{3}{32}\right)\left(\|\partial^{m-1}\nabla  u\|^2_{H^1}+\|\partial^{m-1}\bar{u}\|^2_{H^1}\right)\\
		&\qquad +C(\|(u,\eta)\|_{H^3}+\|(u,\eta)\|^2_{H^3})\| \partial^{m-1}\nabla \eta\|^2_{L^2}.
		\end{split}
\end{equation*}
\end{proof}

\begin{theorem}\label{tt decay}
Let the assumptions of Theorem \ref{tt result} be satisfied. If we further assume $\|(u_0,\eta_0)\|_{H^{-s}}\leq \epsilon_1$ for some $\epsilon_1=\epsilon_1(s)>0$, where $0<-\beta(d,m)\leq s<\frac{d}{2}$ with $\beta(d,m)$ given as in \eqref{betadm}, then there exists a constant $C>0$ independent of $t$ such that 
	\begin{equation*}
	\|\Lambda^l (u, \eta)(t)\|_{L^2}\leq \frac{C}{(1+t)^\frac{s+l}{2}}
\end{equation*}
for $-s<l \leq m-1$.
\end{theorem}
\begin{proof}[Proof of Theorem \ref{tt decay}]We choose sufficiently small $\delta_0\ll1$ such that 
\begin{equation*}
C_0 \|\partial^{m-1} (u, \eta)\|_{H^1}^2  \leq \|\partial^{m-1} (u, \eta)\|_{H^1}^2 +\delta_0\int \partial^{m-1} u\cdot\nabla\partial^{m-1}\eta\, dx \leq \frac1{C_0}\|\partial^{m-1} (u, \eta)\|_{H^1}^2
\end{equation*}	
for some $C_0 > 0$. For sufficiently small $\| (u,\eta)\|_{H^3\cap \dot{H}^{-s}} \ll1$, it follows from Lemmas \ref{hypo decay}, \ref{hypo decay1}, and \ref{ttdecay lem1} that
\begin{align*}
	&\frac{d}{dt}\left(\| (u, \eta)\|_{H^3\cap\dot{H}^{-s}}^2+\delta_0\sum^{2}_{k=0}\int \partial^ku\cdot\nabla\partial^k\eta \,dx+\delta_0\int \Lambda^{-s}u\cdot\nabla\Lambda^{-s}\eta \,dx\right)\\
	&\quad +\frac{1}{2}\delta_0\|\nabla\eta\|^2_{H^{3}\cap \dot{H}^{-s}}+\|\nabla u\|^2_{H^3\cap\dot{H}^{-s}}+\|\bar{u}\|^2_{H^3\cap \dot{H}^{-s}}\leq  0
\end{align*}
and
\begin{equation*}
	\begin{split}
		&\frac{d}{dt}\left(\frac{1}{2}\|\partial^{m-1}( u, \eta)\|^2_{H^1}+\delta_0\int \partial^{m-1} u\cdot\partial^{m-1} \nabla\eta \,dx\right)+\frac{1}{2}\delta_0\|\nabla\partial^{m-1} \eta\|^2_{L^2}+\frac{1}{2}\|\nabla \partial^{m-1} u\|^2_{H^1}\leq 0.
	\end{split}
\end{equation*}
Then, by using almost the same argument as in the proof of Theorem \ref{parabolic decay}, we get 
\begin{equation*}
	\begin{split}
		&\frac{d}{dt}\left(\frac{1}{2}\|\partial^{m-1} (u, \eta)\|^2_{H^1}+\delta_0\int \partial^{m-1} u\cdot\partial^{m-1} \nabla\eta \,dx\right)\\
		&\quad +C\left(\frac{1}{2}\|\partial^{m-1} (u, \eta)\|^2_{H^1}+\delta_0\int \partial^{m-1} u\cdot\partial^{m-1} \nabla\eta \, dx\right)^\frac{m+s}{m-1+s}\leq 0.
	\end{split}
\end{equation*}
This yields 
\begin{equation*}
	\|\partial^{m-1}(u,\eta)(t)\|_{H^1}\leq \frac{C}{(t+1)^\frac{m-1+s}{2}},
\end{equation*}
where $-\frac{d}{2} <-s\leq \beta(d,m)$. Similarly as before, we finally use the interpolation inequality to conclude 
\begin{equation*}
	\|\Lambda^l (u,\eta)\|_{L^2}\leq \|\Lambda^{m-1}(u,\eta)\|^\frac{l+s}{m-1+s} _{L^2}\|\Lambda^{-s}(u,\eta)\|^\frac{m-1-l}{m-1+s} _{L^2}\leq \frac{C}{(t+1)^\frac{l+s}{2} }
\end{equation*} 
for $-s<l \leq m-1$. This completes the proof.
\end{proof}

%
%
%
%
%
%
%
\section*{Acknowledgments}
 
The work of Y.-P. Choi and W. Lee are supported by NRF grant no. 2022R1A2C1002820.

%
%
%
%
%
%
%


\end{document}